\documentclass{amsart}
\usepackage[margin=29mm]{geometry}
\usepackage{math}
\usepackage{mathtools,tikz,mathrsfs,amsrefs}
\usepackage[mathlines]{lineno}
\usetikzlibrary{cd}
\numberwithin{equation}{section}
\newcommand\rr{{\mathfrak r}}
\newcommand\ii{{\mathfrak i}}
\renewcommand\Lie{{\mathscr L}}
\newcommand\lder{{\mathbb L}}

\begin{document}
%\linenumbers

\title{Group and Lie algebra filtrations and homotopy groups of spheres}
\author{Laurent Bartholdi}
\address{Mathematisches Institut, Georg-August Universit\"at zu G\"ottingen}\email{laurent.bartholdi@gmail.com}
\thanks{The first author is supported by the ``@raction'' grant ANR-14-ACHN-0018-01.}

\author{Roman Mikhailov}
\address{Laboratory of Modern Algebra and Applications, St. Petersburg State University, 14th Line, 29b,
Saint Petersburg, 199178 Russia and St. Petersburg Department of
Steklov Mathematical Institute} \email{rmikhailov@mail.ru}
\thanks{The second author is supported by the grant of the Government of the Russian Federation for the state support of scientific research carried out under the supervision of leading scientists, agreement 14.W03.31.0030 dated 15.02.2018.}

\date{January 9, 2021}

\dedicatory{To the memory of John R. Stallings, 1935--2008}

\begin{abstract}
  We establish a bridge between homotopy groups of spheres and commutator calculus in groups, and solve in this manner the ``dimension problem'' by providing a converse to Sjogren's theorem: every abelian group of bounded exponent can be embedded in the dimension quotient of a group.

  This is proven by embedding for arbitrary $s,d$ the torsion of the homotopy group $\pi_s(S^d)$ into a dimension quotient, via a result of Wu. In particular, this invalidates some long-standing results in the literature, since for every prime $p$, there is some $p$-torsion in $\pi_{2p}(S^2)$ by a result of Serre. We explain in this manner Rips's famous counterexample to the dimension conjecture in terms of the homotopy group $\pi_4(S^2)=\mathbb Z/2\mathbb Z$.

  We finally obtain analogous results in the context of Lie rings: for every prime $p$ there exists a Lie ring with $p$-torsion in some dimension quotient.
\end{abstract}
\maketitle

%%%%%%%%%%%%%%%%%%%%%%%%%%%%%%%%%%%%%%%%%%%%%%%%%%%%%%%%%%%%%%%%
\section{Introduction}
The fundamental problem of combinatorial group theory can be phrased as:
``\emph{Given a group $G$ presented as a quotient of a free group,
  what can be said of quotients of $G$ itself?}''. Of fundamental
importance are those quotients produced by universal constructions;
prominently the maximal nilpotent quotients $G/\gamma_n(G)$, and more
generally $G/N$ for $N\triangleleft G$ obtained from $G$ using the
elementary operations of product, intersection, commutation and power.

An altogether different family of quotients arise from associative
algebra.  Every group $G$ naturally embeds in its \emph{group ring}
$\Z G$, leading to images of $G$ in the quotients $\Z G/\varpi^n$ by
powers of the augmentation ideal; and more generally
$G/G\cap(1+\mathfrak N)$ for ideals $\mathfrak N\triangleleft\Z G$
obtained from $\varpi$ using the elementary operations of product,
intersection, sum, and scalar multiple. These quotients play a crucial
role as the receptacle for numerous topological invariants, such as
Milnor's link invariants~\cites{Fox:fdc1,Fox:fdc4,MR1470727}.

A key insight of Magnus~\cite{Magnus:1935} was that the filtrations
$\gamma_n(G)$ of $G$ and $\varpi^n$ of $\Z G$ are deeply related:
defining $\delta_n(G)\coloneqq G\cap(1+\varpi^n)=\ker(G\to\Z G/\varpi^n)$ the
$n$th \emph{dimension subgroup}, one has $\gamma_n(G)\le\delta_n(G)$,
and the \emph{dimension problem} asks to understand when
$\delta_n(G)=\gamma_n(G)$. In that same paper, Magnus showed that if
$F$ is a free group then $\delta_n(F)=\gamma_n(F)$ for all $n$.

It was claimed on numerous occasions~\cites{Cohn:1952,Losey:1960,MKS}
that $\delta_n(G)=\gamma_n(G)$ holds for all $n$ and all groups
$G$. It is relatively easy to prove $\delta_n(G)=\gamma_n(G)$ for
$n\le 3$, but a counterexample was found by Rips~\cite{Rips:1972},
with $\delta_4(G)/\gamma_4(G)=\Z/2\Z$. Nevertheless, the quotient
$\delta_n(G)/\gamma_n(G)$ is always abelian, and Sjogren bounded its
exponent by a function of $n$ only~\cite{Sjogren:1979}; the
following claim even stood in the literature:
\newtheorem*{mainclaim}{``Theorem''}
\begin{mainclaim}[Gupta~\citelist{\cite{Gupta:1996}*{\S4}\cite{Gupta:2002}*{Theorem~2.2}}]
  For all $n$ and all groups $G$ one has
  $\delta_n(G)^2\le\gamma_n(G)\le \delta_n(G)$.
\end{mainclaim}

\newtheorem{mainthm}{Theorem}\renewcommand\themainthm{\Alph{mainthm}}
\noindent Our main result is that this cannot hold, and Sjogren's result alluded to above is essentially optimal:
\begin{mainthm}[See Theorem~\ref{thm:groupmain}]\label{thm:main}
  For every abelian group $H$ of bounded exponent, there exists a
  group $G$ one of whose quotients $\delta_n(G)/\gamma_n(G)$ contains
  $H$ as a subgroup. In particular $p$-torsion may appear in
  $\delta_n(G)/\gamma_n(G)$ for all primes $p$.
\end{mainthm}
By Pr\"ufer's theorem, we reduce to $H$ cyclic. It is known that every
finite cyclic group appears as a subgroup of $\pi_s(S^d)$ for some
$s,d$. For these parameters, we construct a group $G$, integer $n$ and
monomorphism
$\textsf{torsion}(\pi_s(S^d))\hookrightarrow\delta_n(G)/\gamma_n(G)$. In
fact, it suffices to embed $\pi_s(S^2\vee S^2)$, since by Hilton's
theorem this group contains $\pi_s(S^d)$ for all $d$. We even
construct explicitly, for $p=2$ and $p=3$, an element of order $p$ in
$\delta_n(G)/\gamma_n(G)$, based on the element of order $p$ in
$\pi_{2p}(S^2)$ discovered by Serre~\cite{Serre}.

Theorem~\ref{thm:main} is thus a converse to Sjogren's theorem: the
best general constraint on the quotients $\delta_n(G)/\gamma_n(G)$ is
precisely that they are abelian of bounded exponent.

Our method is in principle applicable in a broad setting, producing
for every $K(\pi,1)$ space $X$ and integer $s$ a group $G$, an integer
$n$ and a homomorphism
\[\textsf{torsion}\big(\pi_{s+1}(\Sigma X)\big)\to\delta_n(G)/\gamma_n(G),\]
which we expect to be injective under mild finiteness conditions on $X$.

Homotopy groups of spheres are so fundamental objects that they
pervade topology, with applications ranging from Brouwer's fixed-point
theorem to Rokhlin's theorem on signatures of spin
$4$-manifolds. Their non-triviality and finiteness (apart from the
$\pi_n(S^n)$ and $\pi_{4n-1}(S^{2n})$) are among the most profound
results of mathematics. Theorem~\ref{thm:main} shows that they are
also tightly linked to a question in pure algebra.

Cohen, Wu and their coauthors revealed deep links between
combinatorial group theory and
homotopy~\cites{Cohen:1993,CohenWu:2001,Wu:2001,Wu:2010,CohenWu:2011,LiWu:2011,HuangWu:2020}. At
the heart of our method is a formula by Wu, expressing homotopy groups
of spheres as quotients between two subgroups of a finitely generated
free group. It is based on simplicial sets and corresponding
simplicial groups, see May's fundamental reference~\cite{May:1967}.
    
Topological methods are inherent to the modern study of group theory,
as witnessed by the monumental treatises by Gromov~\cite{Gromov:1991},
Bridson\&Haefliger~\cite{Bridson-Haefliger:1999} and
Geoghegan~\cite{Geoghegan:2008}.
Stallings~\cite{Stallings:1975}*{page~117}, in a programme carried out
by Sjogren~\cite{Sjogren:1979}, already recognized the value of
homological arguments towards studying dimension quotients.

Nonetheless, Theorem~\ref{thm:main} is the first instance of a
classical problem in algebra that is solved using higher algebraic
topology, in effect harnessing the powerful instruments of Steenrod
algebra and spectral sequences, notably the Adams, Curtis and May
spectral sequences.

\subsection{History of the dimension problem}
The dimension problem has a long history, starting with Magnus's
investigation of the lower central series of a free group, and its
associated Lie algebra~\cite{Magnus:1935}; he showed with Witt that
the dimension property holds for free groups,
see~\cites{Magnus:1937,Witt:1937}.  For a small subset of the
literature we refer
to~\cites{Tahara:1977a,Passi:1979,Gupta:1987,Gupta:1988,Mikhailov-Passi:2009}. Remarkably,
incorrect proofs of the dimension problem appeared more than once, by
Cohn~\cite{Cohn:1952}, Losey~\cite{Losey:1960} --- Lyndon remarks
dryly, in his MathSciNet review, ``\textit{The main content of this
  paper is another incomplete proof that the (integral) dimension
  subgroups of an arbitrary group are the terms of its lower central
  series}'' --- and even Magnus himself~\cite{MKS}*{Theorem~5.15(i)}!

We note that if one replaces the ring $\Z$ by a field, then there is
an elegant and elementary description of the corresponding dimension
subgroup, depending only on the field's characteristic;
see~\cites{jennings:gpringnilp,lazard:nilp}.

The dimension problem is quantitatively studied in terms of the
quotients $\delta_n(G)/\gamma_n(G)$ called \emph{dimension
  quotients}. Gupta and Kuzmin proved in~\cite{Gupta-Kuzmin:1992} that
they are all abelian, and Sjogren proved in~\cite{Sjogren:1979} that
they have finite exponent, bounded by a function of $n$ only: there
exists a minimal $s(n)\in\N$ (at most $(n!)^n$) such that
$\delta_n(G)^{s(n)}\subseteq\gamma_n(G)\le\delta_n(G)$ holds for all
groups $G$.

In that terminology, one has $s(1)=s(2)=s(3)=1$, and Rips's example
implies $2\mid s(4)$, which generalizes to $2\mid s(n)$ for all
$n\ge4$. Passi~\cite{Passi:1968} gave $s(4)=2$, and
Tahara~\cite{Tahara:1981} gave $s(5)\in\{2,6\}$.

This can be improved in the case of metabelian groups: Gupta proved
in~\cite{Gupta:1991} that $s(n)$ is a power of $2$. He then claimed
that $s(n)$ is a power of $2$ for all groups --- a proof is published
in~\cite{Gupta:2002}; and even that it may be improved to $s(n)=2$ for
all $n\ge4$, see~\cite{Gupta:1996}. However, many parts of his
arguments were never fully understood.

Now our main result, stated above, shows that the function $s(n)$ is
unbounded, and its values cannot even be constrained to a finite
collection of primes.

\subsection{Lie rings}
An variant of the dimension problem may be asked for Lie rings;
namely, Lie algebras over $\Z$. Every Lie ring $A$ embeds in its
universal enveloping algebra $U(A)$, which also admits an augmentation
ideal. The dimension subrings are defined analogously by
$\delta_n(A)=A\cap\varpi^n$, see~\cite{Bartholdi-Passi:2015}. Again
$\delta_n(A)=\gamma_n(A)$ when $n\le3$, and there is a Lie ring $A$
with $\delta_n(A)/\gamma_n(A)=\Z/2$. Sjogren's bound also holds for
Lie rings~\cite{Sicking:2020}, and many details are simpler in the
category of Lie rings.

Even though we are not aware of any direct construction of a group
from a Lie ring or \emph{vice versa} that preserves dimension
quotients, it often happens that a presentation involving only powers
and commutators, which may therefore be interpreted either as group or
Lie algebra presentation, yields isomorphic dimension quotients.

To give a quick taste of dimension quotients in Lie rings, we
reproduce first an example due to Pierre Cartier of a Lie algebra over
a commutative ring $\Bbbk$ \emph{not} embedding in its universal
envelope~\cite{Cartier:1958}: consider
$\Bbbk=\mathbb F_2[x_0,x_1,x_2]/(x_0^2,x_1^2,x_2^2)$, and
\[A=\langle e_0,e_1,e_2\mid x_0 e_0+x_1 e_1+x_2 e_2=0\rangle\text{ qua $\Bbbk$-Lie algebra}.
\]
Then $\alpha\coloneqq x_0x_1[e_0,e_1]+x_0x_2[e_0,e_2]+x_1x_2[e_1,e_2]$
is non-trivial in $A$, but in any associative algebra it maps to
$(x_0 e_0+x_1 e_1+x_2 e_2)^2=0$.

Rips' example, or rather its Lie algebra
variant~\cite{Bartholdi-Passi:2015}*{Theorem~4.7}, is of a similar
spirit. In $\Bbbk=\Z$ one can of course not choose $x_i$ nilpotent;
but one may choose $x_i$ a large power of $2$ and impose relations
that guarantee that elements with large $2$-valuation are mapped far
in the lower central series: set $x_i=2^{2+i}$ and consider
\begin{equation}\label{eq:rips}
  \begin{split}
    \MoveEqLeft[4] A=\langle e_0,e_1,e_2,\dots\mid 2^{2i+2} e_i\in\gamma_2\text{ for all }i\in\{0,1,2\},\\
    &x_j x_k e_i\pm x_i x_k e_j\in 2^{2k+2}\gamma_2+\gamma_3\text{ for all }\{i,j,k\}=\{0,1,2\}\rangle
  \end{split}
\end{equation}
with the element $\alpha=\sum_{0\le i<j\le2}x_i x_j[e_i,e_j]$. Then
the relations imply
$\alpha\in
A\cap(\gamma_2(A)\cdot\gamma_2(A)+A\cdot\gamma_3(A))\subseteq\delta_4(A)$,
while it is easy to make choices of elements in $\gamma_2$ and
$2^{2+2k}\gamma_2+\gamma_3$ that yield, by direct computation, that
$\alpha$ has a non-trivial image in the quotient $A/\gamma_4(A)$.

It is even possible to write a $3$-related Lie algebra, based
on~\cite{Mikhailov-Passi:2009}*{Example~2.3}, that
satisfies~\eqref{eq:rips} and
$\alpha\in\delta_4(A)\setminus\gamma_4(A)$:
\[A=\langle e_0,e_1,e_2,z\mid 2^2e_0=[z,e_1+2e_2],
  2^4e_1=[z,-e_0+4e_2], 2^6e_2=[z,-2e_0-4e_1]\rangle
\]
with as before $\alpha=2^5[e_0,e_1]+2^6[e_0,e_2]+2^7[e_1,e_2]$.

This Lie algebra presentation may also be interpreted as a group presentation,
\[G=\langle e_0,e_1,e_2,z\mid e_0^4=[z,e_1]\cdot[z,e_2]^2,
  e_1^{16}=[z,e_0]^{-1}\cdot[z,e_2]^4,
  e_2^{64}=[z,e_0]^{-2}\cdot[z,e_1]^{-4}\rangle,
\]
in which the element
$\alpha=[e_0,e_1]^{32}[e_0,e_2]^{64}[e_1,e_2]^{128}$ belongs to
$\delta_4(G)\setminus\gamma_4(G)$.

\subsection{Homotopy groups of the two-sphere: main statement and sketch of proof}
According to a result of Wu~\cites{Wu:2001,Ellis-Mikhailov:2010}, we may
express the homotopy groups of spheres $\pi_{s+1}(S^2)$ as a quotient
of two normal subgroups in a free group. More precisely, write
$F_s=\langle x_0,\dots,x_s\mid x_0\cdots x_s=1\rangle$ a free group of
rank $s$ with one redundant generator, and for $i=0,\dots,s$ let $R_i$
denote the normal closure of $x_i$ in $F_s$. We write iterated
commutators as left-normed:
$[x_1,x_2,\dots,x_d]=[[\cdots[x_1,x_2],\dots],x_d]$, and denote by
$\Sigma_{s+1}$ the symmetric group on $\{0,\dots,s\}$. Then
\begin{equation}\label{eq:wu}
  \frac{R_0\cap\dots\cap R_s}{\prod\limits_{\rho\in\Sigma_{s+1}}[R_{\rho(0)},\dots, R_{\rho(s)}]}\simeq\pi_{s+1}(S^2).
\end{equation}
We can now state more precisely the main step towards our result:
\let\oldthethm=\thethm\def\thethm{\oldthethm\cprime}
\begin{thm}\label{thm:groupmain0}
  Given an integer $s\ge3$, there is for all $n$ large enough a group
  $G$ and a set-wise map $F_s\to G$ inducing via~\eqref{eq:wu} an
  injective homomorphism
  \[\pi_{s+1}(S^2)\hookrightarrow \delta_n(G)/\gamma_n(G).\]
  In particular the exponent of $\delta_n(G)/\gamma_n(G)$ is divisible
  by that of $\pi_{s+1}(S^2)$.
\end{thm}
We note that for $s=3$ this result produces a variant of Rips's
example~\cite{Rips:1972}, ``explaining'' the $2$-torsion in
$\delta_4(G)/\gamma_4(G)$ as that of $\pi_4(S^2)$;
see~\S\ref{ss:examples2}.

An analogous result holds in the realm of Lie algebras. There,
starting with a free Lie ring
$L_s=\langle x_0,\dots,x_s\mid x_0+\dots+x_s=0\rangle$, define
analogously ideals $I_i=\langle x_i\rangle^{L_s}$; then
\begin{equation}\label{eq:liewu}
  \frac{I_0\cap\dots\cap I_s}{\prod\limits_{\rho\in\Sigma_{s+1}}[I_{\rho(0)},\dots, I_{\rho(s)}]}\simeq\bigoplus_{i\ge1}E_{i,s}^1,
\end{equation}
the $s$th column of the lower central series spectral sequence for
$S^2$. Our result, for Lie algebras, states:
\begin{thm}\label{thm:liemain0}
  Given an integer $s\ge3$, there is for all $n$ large enough a Lie ring
  $A$ and a linear map $L_s\to A$ inducing via~\eqref{eq:liewu} an
  injective homomorphism
  \[\bigoplus_{i\ge1}E_{i,s}^1\hookrightarrow \delta_n(A)/\gamma_n(A).\]
  In particular the exponent of $\delta_n(A)/\gamma_n(A)$ is divisible
  by all primes appearing in the order of
  $\bigoplus_{i\ge1}E_{i,s}^1$.
\end{thm}
\let\thethm=\oldthethm\setcounter{thm}{0}
The spectral sequence $E_{*,s}^*$ converges to $\pi_{s+1}(S^2)$, so
for $s=2p-1$ there is an order-$p$ term $\alpha_p\in E_{s+1,s}^1$ that
survives as the Serre element in $\pi_{s+1}(S^2)$.  We are able to
write it explicitly in terms of a shuffle product.

In fact, the groups $G$ and Lie algebras $A$ appearing in
Theorems~\ref{thm:groupmain0},\ref{thm:liemain0} can be written quite
concretely. In a group or Lie algebra presentation, we introduce the
following notation: for $d\in\N$, when we write a generator $x^{(d)}$
of degree $d$ we mean a list of generators $x_1,\dots,x_d$; when $x$
appears in a relator, it is a shorthand for the left-normed iterated
commutator $x\coloneqq[x_1,\dots,x_d]$ of the generators
$x_1,\dots,x_d$. Thus `$\langle x_1^{(2)},y^{(3)}\mid [x_1,y]\rangle$'
is shorthand for
`$\langle
x_{1,1},x_{1,2},y_1,y_2,y_3\mid[[x_{1,1},x_{1,2}],[y_1,y_2,y_3]]\rangle$'. The
following results, which precise Theorems~\ref{thm:groupmain0}
and~\ref{thm:liemain0}, will respectively be proven
in~\S\ref{ss:groupproof} and~\S\ref{ss:lieproof}.

\begin{thm}\label{thm:groupmain}
  Given an integer $s\ge3$, there are integers $e,c_0,\dots,c_s$ and
  $n=c_0+\dots+c_s$ such that, in the group
  \[G=\left\langle\begin{array}{ll} x_0,\dots,x_s,y_0^{(c_0)},\dots,y_s^{(c_s)},\\(r_w)_{w\in\langle x_0,\dots,x_s\rangle}\end{array}\middle|\begin{array}{ll} x_0\cdots x_s=1,x_i^{e^{n c_i}}=y_i\text{ for }i=0,\dots,s,\\ r_w^{e^n}=w\text{ for all }w\in\langle x_0,\dots,x_s\rangle\end{array}\right\rangle
  \]
  the map
  $\iota\colon w(x_0,\dots,x_s)\in F_s\mapsto w(x_0,\dots,x_s)^{e^{n^2}}$
  induces via~\eqref{eq:wu} an injective homomorphism
  \[\overline\iota\colon\pi_{s+1}(S^2)\hookrightarrow\delta_n(G)/\gamma_n(G).\]
\end{thm}
Note that only a finite number of roots $r_w$ of group elements are
required, though it seems messy to specify exactly which ones.

\begin{thm}\label{thm:liemain}
  Given an integer $s\ge3$, there are integers $e,c_0,\dots,c_s$ and
  $n=c_0+\dots+c_s$ such that, in the Lie ring
  \[A=\langle x_0\dots,x_s,y_0^{(c_0)},\dots,y_s^{(c_s)}\mid x_0+\dots+x_s=0, e^{c_i}x_i=y_i\text{ for }i=0,\dots,s\rangle
  \]
  the linear map $\iota\colon w(x_0,\dots,x_s)\in L_s\mapsto e^n w(x_0,\dots,x_s)\in A$
  induces via~\eqref{eq:liewu} an injective homomorphism
  \[\overline\iota\colon\bigoplus_i E_{i,s}^1\hookrightarrow\delta_n(A)/\gamma_n(A).\]
\end{thm}

The constants $e$ and $c_i$ are somewhat explicit, based on the
exponent of $\pi_{s+1}(S^2)$ and the connectivity of certain
simplicial groups. We have determined tighter values for $p=2$ and
$p=3$, see~\S\ref{ss:examples}.

Here is a sketch of the proof in the group case; the Lie algebra case
is essentially the same, and slightly simpler. The first claim follows
from the fact that $G$ is a free product with amalgamation. The second
claim splits in three parts:
$\iota(\prod_\rho[R_{\rho(0)},\dots,R_{\rho(s)}])\le\gamma_n(G)$,
which follows from standard commutator calculus;
$\iota(R_0\cap\dots\cap R_s)\le\delta_n(G)$, which boils down to two
ingredients: the Hurewicz homomorphism, see
Proposition~\ref{prop:groupassoc}, connecting the group and
associative algebra universes, and the presence of roots $r_w$ of
elements $w\in\langle x_0,\dots,x_s\rangle$ in $G$; and the last part,
$\iota^{-1}(\gamma_n(G))\cap R_0\cap\dots\cap
R_s\le\prod_\rho[R_{\rho(0)},\dots,R_{\rho(s)}]$.

For this last part, we first invoke Curtis' connectedness
theorem~\cite{Curtis:1965}, from which there is an integer $k$ such
that
$\gamma_k(F_s)\cap R_0\cap\dots\cap
R_s\le\prod_\rho[R_{\rho(0)},\dots,R_{\rho(s)}]$. It therefore
suffices to prove
$\iota^{-1}(\gamma_n(G))\cap R_0\cap\dots\cap
R_s\le\prod_\rho[R_{\rho(0)},\dots,R_{\rho(s)}]\gamma_k(F_s)$.

We then use the particular form of the presentation of $G$: consider
an element of $F_s$, identified with its image in $G$, and write it in
terms of commutators $g=[z_1,\dots,z_j]$ with each
$z_i\in\{x_0,\dots,x_s\}$, for some $j<k$. This element $g$ seemingly
defines an element of $\gamma_j(G)$, but in $G$ it may be rewritten,
in the presence of sufficiently high powers of $e$, into a commutator
of higher weight by replacing some $x_i$ by the corresponding
$y_i$. If the $c_i$ are chosen such that $c_0\ge k$ and
$c_i/c_{i-1}\ge k$ for all $i$, then in order for $g$ to belong to
$\gamma_n(G)$ either each $x_i$ must have been replaced at least once
by a $y_i$, so the commutator had to belong to some
$[R_{\rho(0)},\dots,R_{\rho(s)}]$, or a larger power of $e$ is
required, and therefore the original term in $F_s$ itself was a power
of $e$. In this manner we obtain
$g\in\prod_\rho[R_{\rho(0)},\dots,R_{\rho(s)}]F_s^e\gamma_k(F_s)$.

In effect, we use two filtrations on $G$, by the lower central series
and by powers of $e$. Only the substitution $x_s\rightsquigarrow y_s$
increases much the degree in the first filtration; but it consumes a
high degree in the second. All the other substitutions
$x_i\rightsquigarrow y_i$ for $i<s$ involve a trade-off between how
much they increase either degree, and each one requires the previous
one. Finally the substitutions $x_i\rightsquigarrow v_i$ or
$y_i\rightsquigarrow w_i$ decrease the first degree too much to be of
any use in attaining $\gamma_n(G)$.

We derive in Theorem~\ref{thm:pins2} an expression for the $p$-torsion
of $\pi_{2p}(S^2)$, first at the level of Lie algebras, namely on the
first page of the Curtis spectral sequence, and deduce in
Proposition~\ref{prop:groupalphap} some properties of its
representation $\widetilde\alpha_p$ as an element of the free group
$F_{2p-1}$. We make use of an explicit form for $p=2$ and $p=3$ to
obtain smaller examples, in particular for $p=2$ we obtain
straightforward constructions, for arbitrary $n\ge4$, of Lie algebras
in which $\delta_n/\gamma_n$ contains $2$-torsion, and for $p=3$ we
obtain a Lie algebra and a group in which $\delta_7/\gamma_7$ contains
$3$-torsion. These examples have also been checked using computer
algebra software.

\subsection{Wedges of spheres}
An analogous statement to~\eqref{eq:wu} holds for wedges of
two-spheres, and even more generally for suspensions of spaces with
contractible universal cover. We restrict ourselves here to the space
$S^2\vee S^2$, which is sufficient to prove Theorem~\ref{thm:main}.

Consider the group $\overline F_s=F_s*F_s$, and identify the generators of
its factors as $x_{i,j}$ for $i=0,\dots,s$ and $j\in\{1,2\}$; set $\overline R_i=\langle x_{i,1},x_{i,2}\rangle^{\overline F_s}$. Then
\begin{equation}\label{eq:wu2}
  \frac{\overline R_0\cap\dots\cap\overline R_s}{\prod\limits_{\rho\in\Sigma_{s+1}}[\overline R_{\rho(0)},\dots,\overline R_{\rho(s)}]}\simeq\pi_{s+1}(S^2\vee S^2).
\end{equation}
Analogously to Theorem~\ref{thm:groupmain}, we have
\begin{thm}\label{thm:groupmain2}
  Given an integer $s\ge3$, there are integers $e,c_0,\dots,c_s$ and
  $n=c_0+\dots+c_s$ such that, in the group
  \[G=\left\langle\begin{array}{ll} x_{0,1},x_{0,2},\dots,x_{s,1},x_{s,2},\\y_{0,1}^{(c_0)},y_{0,2}^{(c_0)},\dots,y_{s,1}^{(c_s)},y_{s,2}^{(c_s)},\\(r_w)_{w\in\langle x_{0,1},x_{0,2},\dots,x_{s,1},x_{s,2}\rangle}\end{array}\middle|\begin{array}{ll} x_{0,1}\cdots x_{s,1}=x_{0,2}\cdots x_{s,2}=1,\\x_{i,j}^{e^{n c_i}}=y_{i,j}\text{ for }i=0,\dots,s,i\in\{1,2\},\\r_w^{e^n}=w\text{ for all }w\in\langle x_{0,1},x_{0,2},\dots,x_{s,1},x_{s,2}\rangle\end{array}\right\rangle
  \]
  the map
  $\iota\colon w(x_{i,j})\in F_s*F_s\mapsto w(x_{i,j})^{e^{n^2}}$
  induces via~\eqref{eq:wu2} an injective homomorphism
  \[\overline\iota\colon\textsf{torsion}\big(\pi_{s+1}(S^2\vee S^2)\big)\hookrightarrow\delta_n(G)/\gamma_n(G).\]
\end{thm}
Note that only a finite number of roots $r_w$ of group elements are
required, though it seems messy to specify exactly which ones. Note
also that the group constructed in Theorem~\ref{thm:groupmain2} is,
apart from the adjunction of roots $r_w$, the free product of two
copies of the group constructed in Theorem~\ref{thm:groupmain}.

There is a Lie algebra analogue to Theorem~\ref{thm:groupmain2}, which
we do not state because it does not seem to have any applications.  In
particular, we do not know whether there exists a Lie algebra such
that one of its dimension quotients contains $\Z/p^2\Z$-torsion for
some prime $p$. Indeed the torsion in the first page of the spectral
sequence converging to $\pi_*(S^2\vee S^2)$ has only prime
orders. There seems to be a fundamental difference, here between
groups and Lie algebras.

%%%%%%%%%%%%%%%%%%%%%%%%%%%%%%%%%%%%%%%%%%%%%%%%%%%%%%%%%%%%%%%% 
\section{Dimension quotients}
We recall in this section some classical facts about dimension
quotients, and their Lie algebra equivalents:
\begin{prop}[\cite{Sjogren:1979} and \cite{Gupta-Kuzmin:1992}]
  For arbitrary group $G$ or Lie algebra $A$, the quotient
  $\delta_n(G)/\gamma_n(G)$, respectively $\delta_n(A)/\gamma_n(A)$,
  is abelian of bounded exponent.
\end{prop}
The ``bounded exponent'' statement is due to Sjogren. In his notation,
let $b_m$ denote the least common multiple of $\{1,\dots,m\}$, and
define integers $a_i^j$ and $c_n$ recursively by
\[a^n_1=1,\qquad a^n_{k+1}=\prod_{i=1}^k a_i^{n-k+i}b_{n-k},\qquad c_n=\prod_{k=1}^{n-1}a_k^k.
\]
Then for any group $G$ we have
$\delta_n(G)^{c_n}\subseteq\gamma_n(G)$. By~\cite{Sicking:2020}, the
same result holds in Lie algebras: for any Lie algebra $A$ we have
$c_n\delta_n(A)\subseteq\gamma_n(A)$.

Gupta and Kuzmin prove even that the quotient
$\delta_n(G)/\gamma_{n+1}(G)$ is abelian;
following~\cite{Mikhailov-Passi:2009} the same result is easily seen
to hold in Lie algebras. We reproduce the argument, in the Lie algebra
case, because of its simplicity:
\begin{proof}[Proof that $\delta_n(A)/\gamma_{n+1}(A)$ is abelian]
  Let $A$ be nilpotent of class $n$; we are to show that $\delta_n(A)$
  is abelian.  Let $M$ be maximal abelian normal in $A$; so $M$ is an
  $A$-module via adjunction. For any $m\in M,x\in\delta_k(A)$ we have
  $[m,x]\in\gamma_{k+1}(M)$, so $\delta_n(A)$ centralizes $M$. Now
  since $M$ is maximal it is self-centralizing, so $\delta_n(A)\le M$
  and therefore is abelian.
\end{proof}

%%%%%%%%%%%%%%%%%%%%%%%%%%%%%%%%%%%%%%%%%%%%%%%%%%%%%%%%%%%%%%%% 
\section{Homotopy groups of spheres}\label{ss:hs}
We describe in this section the group-theoretic and
Lie-algebra-theoretic formulations of homotopy groups of spheres. They
will be essential in the proofs of Theorems~\ref{thm:liemain}
and~\ref{thm:groupmain}. We use `$n$' in this section for what is
written `$s$' in the rest of the text, to avoid confusion with the
degeneracies $s_i$ in simplicial objects.

\subsection{Groups} Fix an integer $n\geq 1$ and let
$F=\langle x_0,\dots, x_n\mid x_0\cdots x_n=1\rangle$ be a free group of
rank $n$. Consider its normal subgroups
\[R_i\coloneqq \langle x_i\rangle^F\text{ for }i=0,\dots,n.
\]
Note that $F$ is the fundamental group of a $2$-sphere with $n+1$
punctures, and $R_i$ contains the conjugacy class of a loop around the
$i$th puncture; the operation of filling-in the $i$th puncture induces
the map $F\to F/R_i$ on fundamental groups.

Denote by $\Sigma_{n+1}$ the symmetric group on $\{0,\dots,n\}$, and
define the symmetric commutator product of the above subgroups by
\[
[R_0,\dots, R_n]_\Sigma\coloneqq\prod_{\rho\in\Sigma_{n+1}}[R_{\rho(0)},\dots, R_{\rho(n)}].
\]
Here and below the iterated commutators are assumed to be left-normalized,
namely $[R_0,R_1,R_2]=[[R_0,R_1],R_2]$ etc.

We view the circle $S^1$ as a simplicial set. Milnor's $F$
construction produces a group complex, having in degree $n$ a free
group on the degree-$n$ objects of $S^1$ subject to a single relation
($s_0^n(*)=1$) and the same boundaries and degeneracies as
$S^1$. According to a formula due to Jie
Wu~\cites{Wu:2001,Ellis-Mikhailov:2010}, considered in the standard basis
of Milnor's $F[S^1]$-construction, homotopy groups of the sphere $S^2$
can be presented in the following way:
\[
\pi_{n+1}(S^2)\simeq \frac{R_0\cap\dots\cap R_n}{[R_0,\dots,R_n]_\Sigma}.
\]

Consider now for $i=0,\dots,n$ the ideals $\rr_i\coloneqq(x_i-1)\Z[F]$ in the
free group ring $\Z[F]$, and their symmetric product
\[
  (\rr_0,\dots,\rr_n)_\Sigma\coloneqq \sum_{\rho\in\Sigma_{n+1}}\rr_{\rho(0)}\cdots \rr_{\rho(n)}
\]
which is also an ideal in $\Z[F]$.
\begin{prop}\label{prop:groupassoc}
  For $n\geq 3$ we have
  $R_0\cap\dots\cap R_n\le F\cap(1+(\rr_0,\dots, \rr_n)_\Sigma)$ when
  considered in $\Z[F]$.
\end{prop}
\begin{proof}
  It is shown in~\cite{MPW} that the quotient
  $\frac{\rr_0\cap\dots\cap \rr_n}{(\rr_0,\dots, \rr_n)_\Sigma}$ can
  be viewed as the $n$th homotopy group of the simplicial abelian
  group $\Z[F[S^1]]$, and the map $F\to \Z[F]$ given by $f\mapsto f-1$
  induces the following commutative diagram
  \[\begin{tikzcd}
      \displaystyle\frac{R_0\cap\dots\cap R_n}{[R_0,\dots,R_n]_\Sigma} \ar[d,equal]\ar[r] & \displaystyle\frac{\rr_0\cap\dots\cap \rr_n}{(\rr_0,\dots,\rr_n)_\Sigma} \ar[d,equal]\\
      \makebox[0mm][r]{$\pi_{n+1}(S^2)={}$}\pi_n(\Omega S^2)\ar[r] & H_n(\Omega S^2).
    \end{tikzcd}
  \]
  The lower map is the $n$th Hurewicz homomorphism for the loop space
  $\Omega S^2$. Since all homotopy groups $\pi_n(\Omega S^2)$ are
  finite for $n\geq 3$, but all homology groups $H_n(\Omega S^2)$ are
  infinite cyclic ($H_*(\Omega S^2)$ is the tensor algebra generated
  by the homology of $S^1$ in dimension
  one~\cite{Bott-Samelson:1953}), we conclude that, for $n\geq 3$, the
  map in the above diagram is zero.
\end{proof}

\subsection{Lie algebras} One obtains an analogous picture in the case
of Lie algebras over $\Z$. The homotopy groups of the simplicial Lie
algebra
\[L[S^1]\coloneqq \bigoplus_i \gamma_i(F[S^1])/\gamma_{i+1}(F[S^1])\]
are equal to the direct sum of terms in rows of the $E^1$-term of the
Curtis spectral sequence
\[
E_{i,j}^1\coloneqq\pi_j\big(\gamma_i(F[S^1])/\gamma_{i+1}(F[S^1])\big)\Longrightarrow \pi_{j+1}(S^2).
\]
The mod-$p$-lower central series spectral sequence is well-studied,
see for example, the foundational paper~\cite{6A}. The integral case
which we consider here has similar properties,
see~\cites{BM,Leibowitz:1972}. Here we will only need elementary
properties of this spectral sequence.

Observe that the $E^1$-page of the above spectral sequence consists of
derived functors $\lder_j$ in the sense of Dold-Puppe, applied to Lie
functors: if $\Lie^i$ denotes the $i$th Lie functor in the category of
abelian groups, then
\[
\pi_j\big(\gamma_i(F[S^1])/\gamma_{i+1}(F[S^1])\big)=\lder_j\Lie^i(\Z,1).
\]
Recall the definition of derived functors. Let $B$ be an abelian
group, and let $F$ be an endofunctor on the category of abelian
groups. For every $i,n\geq 0$ the derived functors of $F$ in the sense
of Dold-Puppe \cite{DP} are defined by
\[
\lder_i F(B,n)=\pi_i(F K P_\ast[n])
\]
where $P_\ast \to B$ is a projective resolution of $B$, and $K$ is the
Dold-Kan transform, inverse to the Moore normalization functor from
simplicial abelian groups to chain complexes. We denote by
$\lder F(B,n)$ the object $F K(P_\ast[n])$ in the homotopy category of
simplicial abelian groups determined by $F K(P_\ast[n])$, so that
$\lder_i F(B,n) = \pi_i(\lder F(B,n))$.

Consider a free Lie algebra $L$ over $\Z$ with generators
$x_0,\dots,x_n$ and relation $x_0+\dots+x_n=0$, and the Lie ideals
\[I_i\coloneqq \langle x_i\rangle^L\text{ for }i=0,\dots,n.
\]
Define their symmetric product by
\[
[I_0,\dots, I_n]_\Sigma=\sum_{\rho\in \Sigma_{n+1}} [I_{\rho(0)},\dots, I_{\rho(n)}].
\]
The same arguments as in the group case imply the Lie analog of Wu's formula:
\[
\frac{I_0\cap\dots\cap I_{n}}{[I_0,\dots, I_{n}]_\Sigma}\simeq \bigoplus_{i\geq 1}E_{i,n}^1=\bigoplus_{i\geq 1} \lder_n\Lie^i(\Z,1).
\]
In fact --- but we shall not need this --- each term $E_{i,n}^1$ may
be singled out by filtering $L$ via its lower central series: one has
\[
  \frac{I_0\cap\dots\cap I_n\cap \gamma_i(L)}{\big([I_0,\dots,I_n]_\Sigma\cap \gamma_i(L)\big)+(I_0\cap\dots\cap I_n\cap \gamma_{i+1}(L))}\simeq E_{i,n}^1.
\]

Consider the universal enveloping algebra $U(L)$, the corresponding ideals $\ii_i\coloneqq x_i U(L)$ in $U(L)$, and their symmetric
product:
\[
(\ii_0,\dots, \ii_n)_\Sigma=\sum_{\rho\in\Sigma_{n+1}}\ii_{\rho(0)}\cdots \ii_{\rho(n)}.
\]
\begin{prop}\label{prop:lieassoc}
  For $n\geq 3$ we have
  $I_0\cap\dots\cap I_n\le L\cap(\ii_0,\dots, \ii_n)_\Sigma$ when
  considered in the universal enveloping algebra.
\end{prop}
\begin{proof}
  Similarly to the group case, the natural map $L\to U(L)$ induces
  \[\begin{tikzcd}
      \displaystyle\frac{I_0\cap\dots\cap I_n}{[I_0,\dots,I_n]_\Sigma} \ar[d,equal]\ar[r] & \displaystyle\frac{\ii_0\cap\dots\cap \ii_n}{(\ii_0,\dots,\ii_n)_\Sigma} \ar[d,equal]\\
      \displaystyle\bigoplus_{i\ge1} E_{i,n}^1\ar[r] & H_n\big(U(L[S^1])\big).
    \end{tikzcd}
  \]
  By~\cite{Schlesinger} the $E_{i,j}^1$-terms of the lower central
  series spectral sequence for $S^2$ are finite for all $j\geq 3$,
  while the universal enveloping simplicial algebra $U(L[S^1])$ has
  infinite cyclic homology groups in all dimensions. It follows that
  the map is $0$.
\end{proof}

%%%%%%%%%%%%%%%%%%%%%%%%%%%%%%%%%%%%%%%%%%%%%%%%%%%%%%%%%%%%%%%% 
\section{Proof of Theorem~\ref{thm:liemain}}\label{ss:lieproof}
We begin with the proof of Theorem~\ref{thm:liemain} on Lie algebras,
since it is slightly simpler, while conceptually similar, to the
corresponding statement for groups.  Let an integer $s\ge3$ be fixed
throughout this section.

\subsection*{First claim: \boldmath $\iota$ exists and is injective}
The assignment $w(x_0,\dots,x_s)\mapsto e^n w(x_0,\dots,w_s)$
naturally defines a linear map $\iota\colon L_s\to A$, since $L_s$'s
only relator holds in $A$. Furthermore, $A$ is an iterated amalgamated
free product, to wit start with
$\langle x_i (0\le i\le s)\mid x_0+\dots+x_s=0\rangle$ and repeatedly
amalgamate, for $i=0,\dots,s$, with
$\langle y_{i,j} (1\le j\le\ell+c_i)\rangle$ along a one-dimensional
subalgebra, so $\iota$ is injective by the normal form of amalgamated
free products, see~\cite{Bokut-Kukin:1994}*{Theorem~4.4.2}.

\subsection*{Second claim: \boldmath $\iota([I_0,\dots,I_s]_\Sigma)\le\gamma_n(A)$}
Recall $n=\deg(y_0)+\dots+\deg(y_s)$. Then
\begin{align*}
  \iota([I_0,\dots,I_s]_\Sigma)
  &\le e^n[\langle x_0\rangle^A,\dots,\langle x_s\rangle^A]_\Sigma = [\langle e^{c_0}x_0\rangle^A,\dots,\langle e^{c_s}x_s\rangle^A]_\Sigma\\
  &= [\langle y_0\rangle^A,\dots,\langle y_s\rangle^A]_\Sigma \le [\gamma_{c_0}(A),\dots,\gamma_{c_s}(A)]_\Sigma\le\gamma_n(A).
\end{align*}

\subsection*{Third claim: \boldmath $\iota(I_0\cap\dots\cap I_s)\le\delta_n(A)$}
\begin{align*}
  \iota(I_0\cap\dots\cap I_s) &= \iota(L_s\cap(\ii_0,\dots,\ii_s)_\Sigma)\text{ by Proposition~\ref{prop:lieassoc}}\\
                              &\le A\cap e^n(\langle x_0\rangle^{U(A)},\dots,\langle x_s\rangle^{U(A)})_\Sigma = A\cap(\langle e^{c_0}x_0\rangle^{U(A)},\dots,\langle e^{c_s}x_s\rangle^{U(A)})_\Sigma\\
  &= A\cap(\langle y_0\rangle^{U(A)},\dots,\langle y_s\rangle^{U(A)})_\Sigma\le A\cap(\varpi^{c_0},\dots,\varpi^{c_s})_\Sigma\le\delta_n(A).
\end{align*}

It follows that $\iota$ induces a map
$\overline\iota\colon(I_0\cap\cdots\cap
I_s)/[I_0,\dots,I_s]_\Sigma\to\delta_n(A)/\gamma_n(A)$, which is a
homomorphism since its domain (and range) are abelian.

\subsection*{Fourth claim: \boldmath $\overline\iota$ is injective}
It is time to specify more precisely the admissible parameters in the
construction of $A$. The parameter $e$ is the exponent of
$\bigoplus_i E_{i,s}^1$, or any multiple thereof. By the Curtis
connectivity theorem~\cite{Curtis:1965}, or more precisely its variant
for Lie algebras, there is an integer $k$ such that
$\gamma_k(F_s)\cap I_0\cap\dots\cap I_s\le[I_0,\dots,I_s]_\Sigma$: let
us quickly sketch the argument. For any connected free simplicial
group $F$, consider the associated Lie algebra
$L(F)=\bigoplus_i L^i(F_{\text{ab}})$. Now
$\gamma_n(L(F))=\bigoplus_{i\ge n} L^i(F_{\text{ab}})$. By Curtis'
theorem, $L^i(F_{\text{ab}})$ is $\log_2i$-connected, namely
$\pi_s L^i(F_{\text{ab}})=0$ for all $s\le \log_2i$. Therefore, fixing
$s$, we get
$\gamma_k(L)\cap I_0\cap\cdots\cap I_s\cap\le[I_0,\dots,I_s]_\Sigma$
for all $k\ge 2^s$. We apply this to $F=F[S^1]$ and its Lie algebra
$L[S^1]$.

We may choose the parameters $c_0,\dots,c_s$ arbitrarily so long as
$c_0\ge k$ and $c_i\ge k c_{i-1}$ for all $i=1,\dots,s$. To fix
matters, let us choose $c_i=k^{i+1}$.

Since $\iota$ is injective, we are to prove
$\iota^{-1}(\gamma_n(A))\cap I_0\cap\dots\cap
I_s\le[I_0,\dots,I_s]_\Sigma$. We note that $I_i\cap e L_s=e I_i$ for
all $i$, so
$I_0\cap\cdots\cap I_s\cap e L_s=e(I_0\cap\cdots\cap I_s)$. By the
choice of $k$ and $e$, it therefore suffices to prove
\[\iota^{-1}(\gamma_n(A)\cap I_0\cap\dots\cap I_s)\le[I_0,\dots,I_s]_\Sigma+e L_s+\gamma_k(L_s).
\]

Consider the free graded Lie algebra
$M=\langle x_0,\dots,x_s,y_0,\dots,y_s\rangle$, in which $y_i$ has
degree $c_i$: it admits a natural surjection
$\pi\colon M\twoheadrightarrow A$. Given
$v\in\langle x_0,\dots,x_s\rangle\le M$, consider the collection of
expressions in $\pi^{-1}(\pi(v))$ that are obtained by replacing, in
an expression of $v$, some terms $x_i$ by the corresponding $y_i$,
adjusting appropriately the power of $e$. We call $v$ \emph{in
  x-form}, and the corresponding expressions obtained by replacing
some $x_i$ by $y_i$ are called in \emph{xy-form}. We also denote by
$\rho$ the natural map
$e^n\langle x_0,\dots,x_s\rangle\subseteq M\twoheadrightarrow L_s$; we
have $\iota\circ\rho=\pi$.

Let us consider $w\in I_0\cap\dots\cap I_s$, and assume
$\iota(w)\in\gamma_n(A)$. We shall write $w=w_0+w_1+w_2$, with
$w_0\in[I_0,\dots,I_s]_\Sigma$ and $w_1\in e L_s$ and
$w_2\in\gamma_k(L_s)$. Now by assumption $\iota(w)$ may be written as
an element $\tilde v\in M$, all of whose terms have degree at least
$n$; we express this in two steps: first, $\iota(w)$ gives rise to an
x-form $v\in M$ by application of the relation $x_0+\dots+x_s=0$; and
then $v$ gives rise to an xy-form $\tilde v$ of $v$ by application of
the other relations, namely replacement of $x_i$ by $y_i$ with
absorption of $e^{c_i}$ in the coefficient. Indeed it follows from the
form of $A$ as an amalgamated free product that the xy-form $\tilde v$
may be obtained from $w$ first by selecting an appropriate x-form
using the relation $x_0+\dots+x_s=0$, and then converting it to
$\tilde v$; thus there is a natural bijection between the summands of
$v$ and $\tilde v$.

Since $M$ is graded and free, we may write $\tilde v$ in a standard
basis of free Lie algebras, such as a selection of left-normed
commutators.  Let us consider in turn all summands of $\tilde v$, a
typical one being of the form
$\tilde\theta\coloneqq[z_1,\dots,z_\ell]$ with all
$z_i\in\{x_0,\dots,x_s,y_0,\dots,y_s\}$; let $\theta$ be the
corresponding monomial in $v$.

If $\ell\ge k$, we put $\rho(\theta)$ in $w_2$. We may therefore from
now on suppose $\ell<k$. On the other hand, say for $i=0,\dots,s$ that
$n_i$ of the $z_j$'s are the generator $y_i$, and that $n_\infty$ of
the $z_j$'s are in $\{x_0,\dots,x_s\}$; then the weight of
$\tilde\theta$ is
\begin{equation}\label{eq:degreeineq}
  n_\infty+n_0 c_0+\dots+n_s c_s\ge n=c_0+\dots+c_s.
\end{equation}
Combining $0\le n_i<k$ for all $i$ with $c_i\ge k c_{i-1}$ and
$c_0\ge k$, we see by unicity of the base-$k$ representation of an
integer that either $n_0=\dots=n_s=1$ or $n_0 c_0+\dots+n_s c_s> n$.

In the former case, each of the $y_0,\dots,y_s$ in $\tilde\theta$ may
be replaced by the corresponding $x_0,\dots,x_s$ to produce a monomial
$\theta$ in $v$ with coefficient multiplied by $e^{c_0+\dots+c_s}$;
and then this summand belongs to $\iota([I_0,\dots,I_s]_\Sigma)$. Add
$\rho(\theta)$ to $w_0$.

In the latter case, replace again all $y_i$ in $\tilde\theta$ by the
corresponding $x_i$ to produce the monomial $\theta$ in $v$ with
coefficient multiplied by $e^{n_0 c_0+\dots+n_s c_s}$. Remembering
that its coefficient is divisible by $e^{n+1}$, add $\rho(\theta)$ to
$w_1$.

We have in this manner expressed $w$ in the required form
$w_0+w_1+w_2$, concluding the proof that $\overline\iota$ is
injective.

%%%%%%%%%%%%%%%%%%%%%%%%%%%%%%%%%%%%%%%%%%%%%%%%%%%%%%%%%%%%%%%%
\section{Proof of Theorem~\ref{thm:groupmain}}\label{ss:groupproof}
The proof of Theorem~\ref{thm:groupmain} follows closely that of the
previous section; the main difference is that the connection between a
group and its group ring is not quite at tight as that between an
algebra and its universal enveloping algebra. We remedy this issue by
adding roots of elements of $\langle x_0,\dots,x_s\rangle$. Note that
only finitely many elements need a root, but we added all for
simplicity of the argument.

Let an integer $s\ge3$ be fixed throughout this section. We begin by
specifying more precisely the admissible parameters in the
construction of $G$. The parameter $e$ is the exponent of
$\pi_{s+1}(S^2)$, or any multiple thereof. By~\cite{Curtis:1965},
there is an integer $k$ such that
$\gamma_k(F_s)\cap R_0\cap\dots\cap R_s\le[R_0,\dots,R_s]_\Sigma$.  We
then choose $c_0,\dots,c_s$ as before subject to
$c_0,c_i/c_{i-1}\ge k$, for example $c_i=k^{i+1}$.

\subsection*{First claim: \boldmath $\iota$ exists and is injective}
The assignment $w(x_0,\dots,x_s)\mapsto w(x_0,\dots,w_s)^{e^n}$
naturally defines a map $\iota\colon F_s\to G$, since $F_s$'s only
relator holds in $G$. Furthermore, $G$ is an iterated amalgamated free
product, to wit start with
$\langle x_0,\dots,x_s\mid x_0\cdots x_s=1\rangle$ and repeatedly
amalgamate, for $i=0,\dots,s$, with
$\langle y_{i,j} (1\le j\le c_i)\rangle$ along a cyclic subgroup. Then
amalgamate, for $w\in\langle x_0,\dots,x_s\rangle$, with
$\langle r_w\rangle$ along a cyclic subgroup.  By the standard normal
form theorem for free products with amalgamation
(see~\cite{Lyndon-Schupp:1970}*{Theorem~IV.2.6}), the map $\iota$ is
injective.

\subsection*{Second claim: \boldmath $\iota([R_0,\dots,R_s]_\Sigma)\le\gamma_n(G)$}
The Lie algebra identities $e[x,y]=[e x,y]=[x,e y]$ do not quite hold
in groups, but in the presence of sufficient roots a close analogue
exists. For $g$ in a group $H$ we denote by $g^H$ the normal closure
of $\langle g\rangle$ in $H$, and by $[g,(c)h]$ the iterated
commutator $[g,h,\dots,h]$ with $c$ copies of `$h$':

\begin{lem}\label{lem:groupps}
  Let $H$ be a group, let $c,e$ be integers, and assume there are
  elements $g\in\gamma_m(H)$ and $h,v\in H$ with $v^{e^c}=h$. Then for
  all $d\in\N$ we have
  \begin{align}
    [g,h^{e^d}] &\in [g,h]^{e^d H}[g,(c+1)v]^H,\label{eq:groupps:1}\\
    [g,h]^{e^d} &\in [g,h^{e^d}]^H[g,(c+1)v]^H.\label{eq:groupps:2}
  \end{align}
\end{lem}
\begin{proof}
  It suffices to prove the statement for $d=1$, since it may be
  applied repeatedly to $[g,h^{e^i}]$, respectively $[g,h]^{e^i}$, for
  $i=1,\dots,d$. We thus restrict ourselves to $d=1$.

  The claims are proven by induction on $c$, the case $c=0$ being
  covered by the first line below. For~\eqref{eq:groupps:1}, write
  $v_1=v^{e^{c-1}}$ and note
  \begin{align*}
    [g,h^e] &= [g,h]\cdot[g,h]^h\cdots[g,h]^{h^{e-1}}\in[g,h]^e[g,h,h]^H;\\
    \text{and }[g,h,h] &=[g,h,v_1^e]\in[g,h,v_1]^{e H}[g,h,(c)v]^H\text{ by induction}\\
            &\le[g,h]^{e H}[g,(c+1)v]^H
  \end{align*}
  since $[g,h,(c)v]\in[g,(c+1)v]^H$. For~\eqref{eq:groupps:2}, note
  \begin{align*}
    [g,h]^e &\in [g,h^e]\cdot[g,h,h]^H\text{ as before};\\
    \text{and }[g,h,h] &= [g,h,v_1^e]\in[g,h,v_1]^{e H}[g,h,(c)v]^H\text{ by~\eqref{eq:groupps:1}}\\
            &\le [g,v_1,h]^{e H}[g,(c+1)v]^H\text{ since $v_1,h$ commute}\\
            &\le [g,v_1,h^e]^H[g,(c+1)v]^H\text{ by induction}\\
            &\le [g,h^e]^H[g,(c+1)v]^H.\qedhere
  \end{align*}
\end{proof}

\noindent Now as in the Lie ring case we have, recalling $n=\deg(y_0)+\dots+\deg(y_s)$,
\begin{align*}
  \iota([R_0,\dots,R_s]_\Sigma)
  &\le ([x_0^G,\dots,x_s^G]_\Sigma)^{e^{n^2}} = [x_0^{e^{n c_0}G},\dots,x_s^{e^{n c_s}G}]_\Sigma\cdot\gamma_n(G)\text{ by~\eqref{eq:groupps:2}}\\
  &= [y_0^G,\dots,y_s^G]_\Sigma\cdot\gamma_n(G) \le [\gamma_{c_0}(G),\dots,\gamma_{c_s}(G)]_\Sigma\cdot\gamma_n(G)\le\gamma_n(G).
\end{align*}

\subsection*{Third claim: \boldmath $\iota(R_0\cap\dots\cap R_s)\le\delta_n(R)$}
Recall that in any group $H$, if $m$ is an integer and $h\in H$ then
\[h^m-1=(h-1+1)^m-1=\sum_{i\ge1}\binom m i(h-1)^i\in m(h-1)+(h-1)^2\Z H.\]
We extend this identity, in the presence of roots, as follows:
\begin{lem}\label{lem:assocps}
  Let $H$ be a group, let $c,e$ be integers, and assume $H$ contains
  elements $h,v$ with $v^{e^c}=h$. Then for all $d\in\N$ we have
  \begin{align}
    h^{e^d}-1 &\in e^d(h-1)\Z H+(v-1)^{c+1}\Z H,\label{eq:assocps:1}\\
    e^d(h-1) &\in (h^{e^d}-1)\Z H+(v-1)^{c+1}\Z H.\label{eq:assocps:2}
  \end{align}
\end{lem}
\begin{proof}
  By applying repeatedly the lemma with $(h^{e^i},v^{e^i})$ for
  $i=0,\dots,d-1$, it is enough to consider $d=1$.
  
  For~\eqref{eq:assocps:1}, we proceed by induction on $c$, noting
  that the case $c=0$ follows from the first line. Setting
  $v_1=v^{e^{c-1}}$,
  \begin{align*}
    h^e-1 &= (h-1+1)^e-1=\sum_{1\le i\le e}\binom e i(h-1)^i\in e(h-1)+(h-1)^2\Z H\\
          &= e(h-1) + (h-1)(v_1^e-1)\Z H\\
          &\le e(h-1) + (h-1)\big(e(v_1-1)+(v-1)^c\Z H\big)\text{ by induction}\\
          &\le e(h-1)\Z H+(v-1)^{c+1}\Z H
  \end{align*}
  since $v-1$ divides $v_1-1$ and $h-1$ and commutes with
  them. For~\eqref{eq:assocps:2}, we also proceed by induction on $c$:
  \begin{align*}
    e(h-1) &= (h^e-1) - \sum_{2\le i\le e}\binom e i(h-1)^i\in(h^e-1)+(h-1)^2\Z H\\
           &= (h^e-1) + (h-1)(v_1^e-1)\Z H\\
           &\le (h^e-1) + (h-1)\big(e(v_1-1)+(v-1)^c\big)\Z H\text{ by~\eqref{eq:assocps:1}}\\
           &\le (h^e-1) + e(h-1)(v_1-1)+(v-1)^{c+1}\Z H\\
           &\le (h^e-1)\Z H+(v-1)^{c+1}\Z H\text{ by induction.}\qedhere
  \end{align*}
\end{proof}

We are now ready to prove
$\iota(R_0\cap\cdots\cap R_s)\le\delta_n(G)$. We have
\begin{align*}
  \iota(R_0\cap\cdots\cap R_s)-1
  &= \iota(F_s\cap(1+(\rr_0,\dots,\rr_s)_\Sigma))-1\text{ by Proposition~\ref{prop:groupassoc}}\\
  &\le(1+(\langle x_0-1\rangle^{\Z G},\dots,\langle x_s-1\rangle^{\Z G})_\Sigma)^{e^{n^2}}-1\\
  &\le e^{n^2}(\langle x_0-1\rangle^{\Z G},\dots,\langle x_s-1\rangle^{\Z G})_\Sigma+\varpi^n\text{ by~\eqref{eq:assocps:1}}\\
  &= (\langle e^{n c_0}(x_0-1)\rangle^{\Z G},\dots,\langle e^{n c_s}(x_s-1)\rangle^{\Z G})_\Sigma+\varpi^n\\
  &= (\langle x_0^{e^{n c_0}}-1\rangle^{\Z G},\dots,\langle x_s^{e^{n c_s}}-1)\rangle^{\Z G})_\Sigma+\varpi^n\text{ by~\eqref{eq:assocps:2}}\\
  &= (\langle y_0-1\rangle^{\Z G},\dots,\langle y_s-1\rangle^{\Z G})_\Sigma+\varpi^n\\
  &\le (\varpi^{c_0},\dots,\varpi^{c_s})_\Sigma+\varpi^n=\varpi^n.
\end{align*}

It follows that $\iota$ induces a map
$\overline\iota\colon(R_0\cap\cdots\cap
R_s)/[R_0,\dots,R_s]_\Sigma\to\delta_n(G)/\gamma_n(G)$, which is a
homomorphism since its domain (and range) are abelian.

\subsection*{Fourth claim: \boldmath $\overline\iota$ is injective}
The argument is essentially the same as in the Lie algebra case, so we
only indicate the differences. We are to prove
\[\iota^{-1}(\gamma_n(G)\cap R_0\cap\cdots\cap R_s)\le[R_0,\dots,R_s]_\Sigma\cdot F_s^e\cdot\gamma_k(F_s).\]
We again start with $w\in R_0\cap\cdots\cap R_s$ and assume
$\iota(w)\in\gamma_n(G)$, and write it as $w=w_0\cdot w_1\cdot w_2$
with $w_0\in[R_0,\dots,R_s]_\Sigma$ and $w_1\in F_s^e$ and
$w_2\in\gamma_k(F_s)$. Again we write $\iota(w)$ as $\tilde v$ in the
free group with generators $x_0,\dots,x_s,y_0,\dots,y_s$, and let $v$
be the corresponding x-form of $\tilde v$. These elements are written
as left-normed commutators, more precisely as left-normed commutators
of generators if their length is $<k$, and as arbitrary commutators of
length $\ge k$. As before, those commutators of length $\ge k$ are
gathered in $w_0$. For each of the remaining ones in $\tilde v$, again
the number of generators $y_i$ in them is denoted by $n_i$, and the
number of other generators is denoted by $n_\infty$;
and~\eqref{eq:degreeineq} still holds.

If all $n_i\ge1$, then we get a term in $w_0$, while if some $n_i=0$
then~\eqref{eq:degreeineq} is strict, and we consider the equation coming from the exponents. Now the other generators are either $x_i$ or their roots $r_{x_i}$; let us write $n_\infty=n_\infty'+n_\infty''$ with $n_\infty'$ the number of generators in $\{x_0,\dots,x_s\}$. Then, in converting a term into its x-form, the exponent gets multiplied by
\[e^{n_0(n c_0)+\dots+n_s(n c_s)-n_\infty'' n)}=e^{n(n_0 c_0+\dots+n_s c_s-n_\infty'')}>e^{n^2},
\]
since all $c_i$ are divisible by $k$ and $n_\infty''<k$. Therefore we
again get a term in $w_1$.

%%%%%%%%%%%%%%%%%%%%%%%%%%%%%%%%%%%%%%%%%%%%%%%%%%%%%%%%%%%%%%%%
\section{Proof of Theorem~\ref{thm:groupmain2}}\label{ss:groupproof2}
We begin by an analogue of Proposition~\ref{prop:groupassoc}. Dropping the ``overlines'' from our notation, consider the group
\[F_s=\langle x_{0,1},x_{0,2},\dots,x_{s,1},x_{s,2}\mid x_{0,1}\cdots
  x_{s,1}=x_{0,2}\cdots x_{s,2}=1=\rangle
\]
and its normal subgroups $R_i=\langle x_{i,1},x_{i,2}\rangle^{F_s}$
for $i=0,\dots,s$. Consider also the ideal $\rr_i=(R_i-1)\Z F_s$.

\begin{prop}\label{prop:groupassoc2}
  For $s\geq 2$ we have
  \[\textsf{torsion}\big(\pi_{s+1}(S^2\vee S^2)\big)\cong\frac{F_s\cap(1+(\rr_0,\dots, \rr_n)_\Sigma)}{[R_0,\dots,R_s]_\Sigma}.\]
\end{prop}
\begin{proof}
  Following~\cite{MPW} the quotient
  $\frac{\rr_0\cap\dots\cap \rr_s}{(\rr_0,\dots, \rr_s)_\Sigma}$ can
  be viewed as the $s$th homotopy group of the simplicial abelian
  group $\Z[F[S^1\vee S^1]]$, and the map $F_s\to \Z[F_s]$ given by $f\mapsto f-1$
  induces the following commutative diagram
  \[\begin{tikzcd}
      \displaystyle\frac{R_0\cap\dots\cap R_s}{[R_0,\dots,R_s]_\Sigma} \ar[d,equal]\ar[r] & \displaystyle\frac{\rr_0\cap\dots\cap \rr_s}{(\rr_0,\dots,\rr_s)_\Sigma} \ar[d,equal]\\
      \makebox[0mm][r]{$\pi_{s+1}(S^2\vee S^2)={}$}\pi_s\big(\Omega(S^2\vee S^2)\big)\ar[r] & H_s\big(\Omega(S^2\vee S^2)\big).
    \end{tikzcd}
  \]
  The lower map is the $s$th Hurewicz homomorphism for the loop space
  $\Omega(S^2\vee S^2)$. By~\cite{Bott-Samelson:1953}, the lower right
  term is the degree-$s$ part of the free associative algebra on two
  generators (the homology of $S^1\vee S^1$). The homotopy group
  $\pi_{s+1}(S^2\vee S^2)$ is the sum of its torsion and torsion-free
  part, and the torsion-free part is the degree-$s$ part of the free
  Lie algebra on two generators (also the homology of $S^1\vee
  S^1$). The lower map, on the torsion-free part, is the natural
  inclusion of the free Lie algebra into the free associative algebra,
  so the kernel of the lower map coincides with the torsion subgroup
  of $\pi_{s+1}(S^2\vee S^2)$.
\end{proof}

Combining Serre's finiteness theorem and Hilton's
theorem~\cite{Hilton:1955}, the torsion of $\pi_{s+1}(S^2\vee S^2)$ is
finite, and in particular has bounded exponent $e$. We may also apply
Curtis's theorem: there is an integer $k$ such that
$\gamma_k(F_s)\cap R_0\cap\cdots\cap R_s=1$. The proof of
Theorem~\ref{thm:groupmain2} then proceeds exactly as that of
Theorem~\ref{thm:groupmain}, with Proposition~\ref{prop:groupassoc2}
used as a replacement of Proposition~\ref{prop:groupassoc}.

%%%%%%%%%%%%%%%%%%%%%%%%%%%%%%%%%%%%%%%%%%%%%%%%%%%%%%%%%%%%%%%%
\section{Proof of Theorem~\ref{thm:main}}\label{ss:mainproof}
Let $H$ be an abelian group of bounded exponent. We begin by
recalling Pr\"ufer's ``first'' theorem~\cite{Prufer:1923}: every
abelian group of bounded exponent is a direct product of cyclic
groups. Now clearly
\[\delta_n\big(\prod_\alpha G_\alpha\big)=\prod_\alpha\delta_n(G_\alpha),\quad
  \gamma_n\big(\prod_\alpha G_\alpha\big)=\prod_\alpha\gamma_n(G_\alpha),
\]
so it suffices to prove Theorem~\ref{thm:main} for cyclic $H$.

We recall next Hilton's theorem~\cite{Hilton:1955}:
\[\pi_{s+1}(S^2\vee S^2)=\bigoplus_d\pi_{s+1}(S^d)\otimes\Lie_d(\Z^2).\]
Therefore in particular every $\pi_{s+1}(S^d)$ is a direct summand of
$\pi_{s+1}(S^2\vee S^2)$.

We finally recall Gray's theorem~\cite{Gray:1969}, proving that the
exponent bound of Cohen-Moore-Neisendorfer is optimal: arbitrary
cyclic groups appear as subgroups of some $\pi_{s+1}(S^d)$.

It follows that for every cyclic group $H$ there is an integer $s$
such that $\pi_{s+1}(S^2\vee S^2)$ contains a copy of $H$. We then
conclude by Theorem~\ref{thm:groupmain2}.

%%%%%%%%%%%%%%%%%%%%%%%%%%%%%%%%%%%%%%%%%%%%%%%%%%%%%%%%%%%%%%%% 
\section{The Serre element in $\pi_{2p}(S^2)$}
Let $p$ be a prime. In this subsection we describe explicitly a copy
of $\Z/p$ in $\pi_{2p}(S^2)$ due to Serre~\cite{Serre}, by computing
its (pre)image $\alpha_p$ in the $E^1$-term of the lower central
spectral sequence associated to $F[S^1]$. There is a single
$(\Z/p)$-term in dimension $2p-1$ of the spectral sequence
\[
\textsf{p-torsion}\left(\frac{I_0\cap\dots\cap I_{2p-1}}{[I_0,\dots, I_{2p-1}]_\Sigma}\right)=\lder_{2p-1}\Lie^{2p}(\Z,1)=\Z/p,
\]
and $\alpha_p$ will be a generator of this subgroup.

\begin{thm}\label{thm:pins2}
  Let $x_i$ for $i=0,\dots,2p-2$ be free generators of a free Lie
  algebra, and consider the following element
  \begin{multline*}
    \alpha_p=\sum_{\kern-1cm\substack{\rho\in\Sigma_{2p-2}\text{ a $2^{p-1}$-shuffle}\\\rho(1)<\rho(3)<\dots<\rho(2p-5)}\kern-1cm}(-1)^\rho[[x_{\rho(0)},x_{2p-2}],[x_{\rho(1)},x_{2p-2}],[x_{\rho(2)},x_{\rho(3)}],\dots, [x_{\rho(2p-4)},x_{\rho(2p-3)}]];
  \end{multline*}
  the sum is taken over all permutations
  $(\rho(0),\dots,\rho(2p-3))\in\Sigma_{2p-2}$ satisfying
  $\rho(0)<\rho(1),\dots,\rho(2p-4)<\rho(2p-3)$ as well as
  $\rho(1)<\rho(3)<\dots<\rho(2p-5)$. Then $\alpha_p$ represents a
  generator of the $p$-torsion in $\lder_{2p-1}\Lie^{2p}(\Z,1)$.
\end{thm}

\newcommand\pair[2]{({#1}\,{#2})}

\begin{proof}
  Consider the free abelian simplicial group $K(\Z,2)$: it has a
  single generator $\sigma$ in degree $2$, and its other generators
  may be chosen to be all iterated degeneracies of $\sigma$. We will
  use the dual notation for generators: for $k>2$ the free abelian
  group $K(\Z,2)_k$ is generated by ordered sequences of two elements
  \[\pair{i_1}{i_2}\coloneqq s_{k-1}\cdots\widehat{s_{i_2}}\cdots\widehat{s_{i_1}}\cdots s_0(\sigma)
  \]
  with $0\le i_1<i_2<k$. For example, $K(\Z,2)_5$ has generators
  \begin{alignat*}{4}
    \pair01&\coloneqq s_4s_3s_2(\sigma), &\; \pair02&\coloneqq s_4s_3s_1(\sigma), &
    \pair03&\coloneqq s_4s_2s_1(\sigma), &\; \pair04&\coloneqq s_3s_2s_1(\sigma),\\
    \pair12&\coloneqq s_4s_3s_0(\sigma), &\; \pair13&\coloneqq s_4s_2s_0(\sigma), &
    \pair14&\coloneqq s_3s_2s_0(\sigma), &\; \pair23&\coloneqq s_4s_1s_0(\sigma),\\
    & & \pair24&\coloneqq s_3s_1s_0(\sigma), &\; \pair34&\coloneqq s_2s_1s_0(\sigma).
  \end{alignat*}

  For $n\ge1$, define the functor $J^n$ as the
  \emph{metabelianization} of the $n$th Lie functor $\Lie^n$. For a
  group $A$, there is a natural epimorphism
  \[
    \Lie^p(A)\twoheadrightarrow J^p(A)
  \]
  with kernel generated by Lie brackets of the form
  $[[*,*],[*,*]]$. The elements of $J^p$ can also be written as linear
  combinations of Lie brackets, namely as elements of the Lie functor
  $\Lie^p$, but there is additional rule which holds in $J^p$
  but not hold in $\Lie^p$ in general:
  \[
    [a_1,a_2,\dots, a_p]=[a_1,a_2,a_{\rho(3)},\dots, a_{\rho(p)}]
  \]
  for arbitrary $a_i$ and permutation $(\rho(3),\dots, \rho(p))$
  of $\{3,\dots, p\}$. For $p=3$, the functors $\Lie^3$ and
  $J^3$ are equal.

  For $n\ge1$, denote by $S^n$ the $n$th symmetric power functor
  \[
    S^n\colon {\sf Abelian\ groups}\to {\sf Abelian\ groups}.
  \]
  For a free abelian group $A$, there is a natural short exact
  sequence~\cite{Schlesinger}*{Proposition 3.2}
  \begin{equation}\label{eq:JtoS}
    0\to J^n(A) \to S^{n-1}(A)\otimes A\to S^n(A)\to 0,
  \end{equation}
  where the left-hand map is given by
  \begin{equation}\label{bra2tens}
    [b_1,\dots, b_n]\mapsto b_2b_3\cdots b_n\otimes b_1-b_1b_3\cdots b_n\otimes b_2\text{ for }b_i\in A.
  \end{equation}

  Applying the functors
  $J^p\hookrightarrow S^{p-1}\otimes id\twoheadrightarrow S^p$ to the
  simplicial abelian group $K(\Z,2n)$, and taking the homotopy groups,
  we get the long exact sequence
  \begin{align*}
    &\pi_{2p n}\big(S^{p-1}K(\Z,2n)\otimes K(\Z,2n)\big)\to \lder_{2p n}S^p(\Z,2n)\to \lder_{2p n-1}J^p(\Z,2n)\to\\
    \to&\pi_{2p n-1}\big(S^{p-1}K(\Z,2n)\otimes K(\Z,2n)\big).
  \end{align*}
  It follows from \cite{DP}*{page~307} that the above sequence has the
  following form:
  \[\begin{tikzcd}
      \pi_{2p n}\big(S^{p-1}K(\Z,2n)\otimes K(\Z,2n)\big)\ar[r] \ar[d,equal] & \lder_{2p n}S^p(\Z,2n)\ar[d,equal]\ar[r] & \lder_{2p n-1}J^p(\Z,2n)\ar[d,equal]\\
      \Z\ar[r,"p"] & \Z\ar[->>,r] & \Z/p.
    \end{tikzcd}
  \]
  By~\cite{Schlesinger}*{Proposition~4.7}, the natural epimorphism
  $\Lie^p\twoheadrightarrow J^p$ gives a natural isomorphism
  of derived functors
  \[
    \lder_{2p n-1}\Lie^p(\Z,2n)\xrightarrow{\simeq} \lder_{2p n-1}J^p(\Z,2n)\simeq \Z/p.
  \]
  Let us first find a simplicial generator of $\lder_{2p}S^p(\Z,2)$. For
  this, we observe that the inclusion of the symmetric power into the
  tensor power $S^p\hookrightarrow {\otimes}^p$ induces an isomorphism of
  derived functors
  \[
    \lder_{2p}S^p(\Z,2)\to \lder_{2p}{\otimes}^p(\Z,2).
  \]
  A simplicial generator of $\lder_{2p}{\otimes}^p(\Z,2)$ can be given by
  the Eilenberg-Zilber shuffle-product theorem. Using interchangeably
  the notation $\rho(i)$ and $\rho_i$, this is the element
  \[
    \sum_{\rho\in\Sigma_{2p}\text{ a $2^p$-shuffle}}(-1)^\rho\pair{\rho_{0}}{\rho_{1}}\otimes \pair{\rho_{2}}{\rho_{3}}\otimes\dots\otimes\pair{\rho_{2p-2}}{\rho_{2p-1}}.
  \]
  It follows immediately from the definition of $2^p$-shuffles that
  the symmetric group $\Sigma_p$, acting by permutation on blocks
  $\{2i,2i+1\}$ of size $2$, acts on $2^p$-shuffles.  A generator of
  $\lder_{2p}S^p(\Z,2)$ can be chosen by keeping only a single element
  per $\Sigma_p$-orbit, and replacing tensor products by symmetric
  products:
  \[
\beta\coloneqq\sum_{\substack{\rho\in\Sigma_{2p}\text{ a $2^p$-shuffle}\\\rho(1)<\rho(3)<\dots<\rho(2p-1)}}(-1)^\rho\pair{\rho_{0}}{\rho_{1}}\cdot\pair{\rho_{2}}{\rho_{3}}\cdots\pair{\rho_{2p-2}}{\rho_{2p-1}}.
  \]
  The conditions imply $\rho(2p-1)=2p-1$. For example, for $p=3$ we
  get the element
  \begin{align*}
    &\pair01\pair23\pair45-\pair01\pair24\pair35+\pair01\pair34\pair25-\pair02\pair34\pair15-\pair02\pair13\pair45\\
    {}+{}&\pair03\pair24\pair15-\pair03\pair14\pair25+\pair12\pair03\pair45-\pair23\pair04\pair15-\pair12\pair04\pair35\\
    {}+{}&\pair12\pair34\pair05-\pair13\pair24\pair05+\pair23\pair14\pair05+\pair02\pair14\pair35+\pair13\pair04\pair25.
  \end{align*}
  Now we lift the element from $S^p K(\Z,2)_{2p}$ to
  $(S^{p-1}K(\Z,2)\otimes K(\Z,2))_{2p}$ in a standard way:
  \[
    \tilde\beta\coloneqq\sum_{\substack{\rho\in\Sigma_{2p}\text{ a $2^p$-shuffle}\\\rho(1)<\rho(3)<\dots<\rho(2p-1)}}(-1)^\rho\pair{\rho_{0}}{\rho_{1}}\cdots\pair{\rho_{2p-4}}{\rho_{2p-3}}\otimes\pair{\rho_{2p-2}}{\rho_{2p-1}}.
  \]

  Observe that we have
  \[d_j\pair{i_1}{i_2} = \begin{cases}
      \pair{i_1}{i_2} & \text{ if }i_2< j,\\
      \pair{i_1}{i_2-1} & \text{ if }i_1< j\le i_2,\\
      \pair{i_1-1}{i_2-1} & \text{ if }j\le i_1
    \end{cases}
  \]
  with the understanding that $\pair ii=0$, that we use the same
  notation $\pair{i_1}{i_2}$ for elements of varying degree, and that
  $d_0\pair{0}{i_2}=0$ and $d_j\pair{i_1}{i_2}=0$ if
  $\deg\pair{i_1}{i_2}=j=i_2-1$. Thus e.g.\ $d_0\pair04=d_5\pair04=0$
  and $d_1\pair04=d_2\pair04=d_3\pair04=d_4\pair04=\pair03$ while
  $d_0\pair23=d_1\pair23=d_2\pair23=\pair12$ and $d_3\pair23=0$ and
  $d_4\pair23=d_5\pair23=\pair23$.

  Clearly $d_0(\tilde\beta)=d_{2p-1}(\tilde\beta)=0$. If $j<2p-2$,
  then we express $\tilde\beta$ as a sum over all possible values of
  $r\coloneqq\rho(2p-2)$ (remembering $\rho(2p-1)=2p-1)$ and obtain
  \[d_j(\tilde\beta)=\sum_{r=0}^{2p-2}(-1)^r\big(d_j(\cdots)\otimes\pair{r}{2p-1}+(\cdots)\otimes d_j\pair{r}{2p-1}\big).
  \]
  Now the sum in $(\cdots)$ is a symmetric
  product similar to $\beta$, but with $p-1$ instead of $p$ factors,
  so $(\cdots)$ is exact. The second terms telescope, so we get
  $d_j(\tilde\beta)=0$ when $j<2p-2$. However, $\tilde\beta$ is not a
  cycle in $S^{p-1}(\mathbb Z,2)\otimes K(\mathbb Z,2)$, because
  $d_{2p-2}(\tilde\beta)$ is not zero: we compute
  \[
    d_{2p-2}(\tilde\beta)=\sum_{\kern-1cm\substack{\rho\in\Sigma_{2p}\text{ a $2^p$-shuffle}\\\rho(1)<\dots<\rho(2p-3)=2p-2>\rho(2p-2)}\kern-1cm}(-1)^\rho\pair{\rho_{0}}{\rho_{1}}\dots\pair{\rho_{2p-4}}{\rho_{2p-3}}\otimes\pair{\rho_{2p-2}}{2p-2}.
  \]
  We use the long exact sequence associated with~\eqref{eq:JtoS} to
  obtain a cycle in $J^p(\Z,2)_{2p-1}$.  The ascending
  $2^p$-shuffles $(\rho(0),\dots,\rho(2p-1))$ appearing in the sum can
  in fact be viewed as $2^{p-1}$-shuffles
  $(\rho(0),\rho(1),\dots,\rho(2p-6),\rho(2p-5),\rho(2p-4),\rho(2p-2))$
  or
  $(\rho(0),\rho(1),\dots,\rho(2p-6),\rho(2p-5),\rho(2p-2),\rho(2p-4))$,
  depending on whether $\rho(2p-2)<\rho(2p-4)$ or not, and in all
  cases completed by the values $(2p-2,2p-1)$. Furthermore, these two
  shuffles come with opposite signs, and can be combined,
  via~\eqref{bra2tens}, into
  \begin{multline}\label{longele}
    \sum_{\kern-1cm\substack{\rho\in\Sigma_{2p-2}\text{ a $2^{p-1}$-shuffle}\\\rho(1)<\cdots<\rho(2p-5)}\kern-1cm}(-1)^\rho[\pair{\rho_{2p-3}}{2p-2},\pair{\rho_{2p-4}}{2p-2},\pair{\rho_{0}}{\rho_{1}},\dots,\pair{\rho_{2p-6}}{\rho_{2p-5}}].
  \end{multline}
  We now consider the simplicial map $K(\Z,2)\to \Lie^2K(\Z,1)$,
  given by $\sigma\mapsto [s_0(\sigma'),s_1(\sigma')],$ where
  $\sigma'$ is the generator of $K(\Z,1)_1$; it is a homotopy
  equivalence of complexes. The abelian group $K(\Z,1)_k$ is
  $k$-dimensional, with generators
  \[x_i\coloneqq s_k\cdots\widehat{s_i}\cdots s_0(\sigma')\] for all
  $0\le i<k$, and we have $\pair{i_1}{i_2}\mapsto[x_{i_1},x_{i_2}]$
  under this homotopy equivalence. Thus $\Lie^*(\Z,2)_{2p-1}$ is a
  free Lie algebra on $2p-1$ generators. There is an induced map
  \[
    \Lie^p K(\Z,2)\to \Lie^p\circ \Lie^2K(\Z,1)\to \Lie^{2p}K(\Z,1)
  \]
  which also is a homotopy equivalence of complexes. The image of the
  element~\eqref{longele} is
  \begin{multline*}
    \sum_{\kern-1cm\substack{\rho\in\Sigma_{2p-2}\text{ a $2^{p-1}$-shuffle}\\\rho(1)<\cdots<\rho(2p-5)}\kern-1cm}(-1)^\rho[[x_{\rho(2p-3)},x_{2p-2}],[x_{\rho(2p-4)},x_{2p-2}],[x_{\rho(0)},x_{\rho(1)}],\dots,[x_{\rho(2p-6)},x_{\rho(2p-5)}]].
  \end{multline*}
  Up to sign and renumbering, this is exactly our element $\alpha_p$.
\end{proof}

Note that we considered, in the beginning of this section, a free Lie
algebra of rank $2p-1$ with $2p$ generators $x_0,\dots,x_{2p-1}$
subject to the relation $\sum x_i=0$. Any choice of $2p-1$ out of
these $2p$ generators yields a free Lie algebra on $2p-1$ generators,
and an expression $\alpha_p$. The point being made is that every such
expression involves one of the generators (here $x_{2p-2}$) twice, and
omits another (here $x_{2p-1}$).

We summarize as follows the properties of the element $\alpha_p$ that
will be useful to us:
\begin{prop}\label{prop:groupalphap}
  For every prime $p$ there is an element $\widetilde\alpha_p$ in the
  free group
  $\langle x_0,\dots,x_{2p-1}\mid x_0\cdots x_{2p-1}=1\rangle$ with
  the properties:
  \begin{itemize}
  \item $\widetilde\alpha_p-1\in(\rr_0,\dots,\rr_{2p-1})_\Sigma$;
  \item $\widetilde\alpha_p\not\in[R_0,\dots,R_{2p-1}]_\Sigma$;
  \item $\widetilde\alpha_p^p\in[R_0,\dots,R_{2p-1}]_\Sigma$.
  \end{itemize}
  Furthermore,
  $\widetilde\alpha_p-1\in([\rr_0,\rr_1],\dots,[\rr_{2p-2},\rr_{2p-1}])_\Sigma$,
  namely in the sum of all $p$-fold associative products of brackets
  of $\rr_i$ in any of the $(2p)!$ orderings.
\end{prop}
\begin{proof}
  The first claim follows from Proposition~\ref{prop:groupassoc},
  since $\alpha_p$ represents an element of $\pi_{2p}(S^2)$.  The
  second claim holds because this element is non-trivial in
  $\pi_{2p}(S^2)$. The third claim holds because it has order $p$ in
  $\pi_{2p}(S^2)$.  The last claim follows from general facts:
  $\lder_i\Lie^n(\Z,1)=0$ for odd $n$, and
  $\lder_i\Lie^{2n}(\Z,1)=\lder_i\Lie^n(\Z,2)$.
\end{proof}

The same statement holds for Lie algebras; we omit the proof.
\begin{prop}\label{prop:liealphap}
  For every prime $p$ there is an element $\alpha_p$ in the free Lie
  algebra
  $\langle x_0,\dots,x_{2p-1}\mid x_0+\cdots+x_{2p-1}=0\rangle$ with
  the properties:
  \begin{itemize}
  \item $\alpha_p\in(I_0,\dots,I_{2p-1})_\Sigma$;
  \item $\alpha_p\not\in[I_0,\dots,I_{2p-1}]_\Sigma$;
  \item $p\alpha_p\in[I_0,\dots,I_{2p-1}]_\Sigma$.
  \end{itemize}
  Furthermore,
  $\alpha_p\in([R_0,R_1],\dots,[R_{2p-2},R_{2p-1}])_\Sigma$.\qed
\end{prop}

\begin{exple}
  Here is an explicit generator of $\pi_4(S^2)=\Z/2$. If we consider
  $p=2$ in Theorem~\ref{thm:pins2}, we have only one $[2]$-shuffle and
  the element $\alpha_2$ is $[[x_0, x_2],[x_1,x_2]]$. Reintroducing
  $x_3=-x_0-x_1-x_2$, we can easily check that
  $\alpha_2\in (I_0,I_1,I_2,I_3)_\Sigma$:
  \begin{align}
    \alpha_2 &\coloneqq [[x_0,x_2],[x_1,x_2]]\notag\\
          &= [x_2,x_3]\cdot[x_0,x_1]+[x_1,x_2]\cdot[x_0,x_3]-[x_0,x_2]\cdot[x_1,x_3].\label{eq:exple2}
  \end{align}
  Applying to it the Dynkin idempotent $u\cdot v\mapsto\frac12[u,v]$
  gives then $2\alpha_2\in[I_0,I_1,I_2,I_3]_\Sigma$.

  It is only slightly harder to write a generator of $\pi_4(S^2)$ in
  the language of groups. We may lift $\alpha_2$ to
  $\widetilde\alpha_2\in F$, the free group
  $\langle x_0,x_1,x_2,x_3\mid x_0\cdots x_3\rangle$, as
  \[\widetilde\alpha_2=[[x_0,x_2],[x_0x_1,x_2]],\]
  since then the Hall-Witt identities give
  $\widetilde\alpha_2=[[x_0,x_2],[x_3^{-1},x_2]^{x_2^{-1}}]=[[x_0,x_2],[x_0,x_2]^{x_1}[x_1,x_2]]=[[x_0,x_2],[x_1,x_2]]\cdot[[x_0,x_2],[[x_0,x_2],x_1^{x_2}]]$
  so $\widetilde\alpha_2\in R_0\cap\cdots\cap R_3$. We have thus
  produced a non-trivial cycle
  $\widetilde\alpha_2\in(R_0\cap R_1\cap R_2\cap
  R_3)/[R_0,R_1,R_2,R_3]_\Sigma$.
\end{exple}

\begin{exple}
  Here is a generator of the $3$-torsion in $\pi_6(S^2)$.  For $p=3$,
  we have six $[2,2]$-shuffles in Theorem~\ref{thm:pins2}:
  \begin{xalignat*}{2}
    (0,1,2,3)&\text{ with sign}=1, & (0,2,1,3)&\text{ with sign}=-1,\\
    (0,3,1,2)&\text{ with sign}=1, & (2,3,0,1)&\text{ with sign}=1,\\
    (1,3,0,2)&\text{ with sign}=-1, & (1,2,0,3)&\text{ with sign}=1.
  \end{xalignat*}
  The element $\alpha_3$ representing $3$-torsion in  $\pi_6(S^2)$ is
  \begin{align*}
    \alpha_3 &\coloneqq [[x_0,x_4],[x_1,x_4],[x_2,x_3]]-[[x_0,x_4],[x_2,x_4],[x_1,x_3]]\\
           & {}+[[x_0,x_4],[x_3,x_4],[x_1,x_2]]+[[x_1,x_4],[x_2,x_4],[x_0,x_3]]\\
           & {}-[[x_1,x_4],[x_3,x_4],[x_0,x_2]]+[[x_2,x_4],[x_3,x_4],[x_0,x_1]].\\
  \end{align*}
  It may be expressed as a sum of $30$ associative products of the
  form $\pm[x_a,x_b]\cdot[x_c,x_d]\cdot[x_e,x_f]$ with
  $\{a,b,c,d,e,f\}=\{0,1,2,3,4,5\}$.
\end{exple}

Again it is possible (but now with considerably more effort) to lift
$\alpha_3$ to a generator of $\pi_6(S^2)$ in terms of free groups. We
return to the notation of simplicial free groups: we consider the free
group $F=\langle z_0,\dots,z_4\rangle$ and normal subgroups
$R_0=\langle z_0\rangle^F$, $R_i=\langle z_{i-1}^{-1}z_i\rangle^F$ for
$i\in\{1,\dots,4\}$ and $R_5=\langle z_4\rangle^F$. In other words, we
set $z_i\coloneqq x_0\cdots x_i$. Here is a lift of $\alpha_3$ to $F$
which defines a simplicial cycle, i.e., which lies in the intersection
$R_0\cap\dots\cap R_5$: it is the product of the following fourteen
elements
\begin{align*}
  \widetilde\alpha_3 &= [[z_0,z_4],[z_2,z_4],[z_1,z_3]^{[z_0,z_4]}]^{-1}\cdot
                     [[z_1,z_4],[z_2,z_4],[z_0,z_3]^{[z_1,z_4]}]\\
                   &{}\cdot  [[z_1,z_4],[z_2,z_3],[z_0,z_4]^{[z_1,z_4]}]^{-1}\cdot
                     [[z_0,z_4],[z_2,z_3],[z_1,z_4]^{[z_0,z_4]}]\\
                   &{}\cdot  [[z_2,z_4],[z_0,z_4],[z_1,z_3]^{[z_2,z_4]}]\cdot
                     [[z_2,z_4],[z_1,z_4],[z_0,z_3]^{[z_2,z_4]}]^{-1}\\
                   &{}\cdot  [[z_2,z_3],[z_1,z_4],[z_0,z_4]^{[z_2,z_3]}]\cdot
                     [[z_2,z_3],[z_0,z_4],[z_1,z_4]^{[z_2,z_3]}]^{-1}\\
                   &{}\cdot  [[z_3,z_4],[z_1,z_4],[z_0,z_2]^{[z_3,z_4]}]\cdot
                     [[z_3,z_4],[z_0,z_4],[z_1,z_2]^{[z_3,z_4]}]^{-1}\\
                   &{}\cdot  [[z_3,z_4],[z_2,z_4],[z_0,z_1]^{[z_3,z_4]}]^{-1}\cdot
                     [[z_1,z_4],[z_3,z_4],[z_0,z_2]^{[z_1,z_4]}]^{-1}\\
                   &{}\cdot  [[z_0,z_4],[z_3,z_4],[z_1,z_2]^{[z_0,z_4]}]\cdot
                     [[z_2,z_4],[z_3,z_4],[z_0,z_1]^{[z_2,z_4]}].
\end{align*}
One can directly check that
$\widetilde\alpha_3$ defines a simplicial cycle and that modulo the
seventh term of the lower central series it represents exactly the
element $\alpha_3$.

\begin{rem}
  We have $\pi_6(S^2)=\Z/3\times\Z/4$, and it is also possible to give
  an explicit generator of the $4$-torsion. In the same notation as
  above, it is
  \begin{align*}
    \widetilde\alpha_4 &= [[[z_3,z_1],[z_3,z_2]],[[z_4,z_0],[z_4,z_2]]]\\
    {}\cdot{}& [[[z_4,z_1],[z_4,z_2]],[[z_3,z_0],[z_3,z_2]]]\\
    {}\cdot{} & [[[[z_4,z_1],[z_4,z_2]],[[z_4,z_0],[z_4,z_2]]],[[z_3,z_2],[z_3,z_1]]]\\
    {}\cdot{} & [[[z_3,z_1],[z_3,z_2]],[[[z_4,z_2],[z_4,z_0]],[[z_4,z_2],[z_4,z_1]]]]\\
    {}\cdot{} & [[[z_3,z_2],[z_3,z_1]],[[z_4,z_2],[z_4,z_0]]]\\
    {}\cdot{} & [[[z_4,z_2],[z_4,z_1]],[[z_3,z_2],[z_3,z_0]]]\\
    {}\cdot{} & [[[z_4,z_2],[z_3,z_0]],[[z_4,z_2],[z_3,z_1]]]\\
    {}\cdot{} & [[[z_4,z_1],[z_3,z_2]],[[z_4,z_2],[z_3,z_0]]]\\
    {}\cdot{} & [[[z_4,z_2],[z_3,z_1]],[[z_4,z_0],[z_3,z_2]]]\\
    {}\cdot{} & [[[z_4,z_0],[z_3,z_2]],[[z_4,z_1],[z_3,z_2]]]\\
    {}\cdot{} & [[[z_4,z_2],[z_3,z_1]],[[z_4,z_3],[z_2,z_0]]]\\
  {}\cdot{} & [[[z_4,z_3],[z_2,z_1]],[[z_4,z_2],[z_3,z_0]]]\\
    {}\cdot{} & [[[z_4,z_3],[z_2,z_0]],[[z_4,z_1],[z_3,z_2]]]\\
    {}\cdot{} & [[[z_4,z_0],[z_3,z_2]],[[z_4,z_3],[z_2,z_1]]]\\
    {}\cdot{} & [[[z_4,z_3],[z_2,z_0]],[[z_4,z_3],[z_2,z_1]]].
  \end{align*}

  Note that $\widetilde\alpha_4^2$ is, up to the symmetric commutator
  $[R_0,\dots,R_5]_\Sigma$, equal to
  \begin{equation}\label{longelement}
    [[[[z_0,z_1],[z_0,z_2]],[[z_0,z_1],[z_0,z_3]]],[[[z_0,z_1],[z_0,z_2]],[[z_0,z_1],[z_0,z_4]]]].
  \end{equation}
  Here is a brief explanation of the origin of $\widetilde\alpha_4$.
  The elements of the $E^1$-page of the spectral sequence can be coded
  by generators of lambda-algebra. Serre elements, which we study,
  correspond to the elements $\lambda_1$. The element
  $\widetilde\alpha_4$ corresponds to $\lambda_2\lambda_1$ of the
  lambda-algebra. The $E_{*,5}^{\infty}$ column of $S^2$ has the
  following non-trivial terms: $E_{8,5}^\infty=\mathbb Z/2$ (generator
  $\lambda_2\lambda_1$), $E_{6,5}^\infty=\mathbb Z/3$ (generator
  $\lambda_1$ for $p=3$), $E_{16,5}^\infty=\mathbb Z/2$ (generator
  $\lambda_1^3$). The 4-torsion in $\pi_6(S^2)$ is glued from two
  terms in $E^\infty$: $\lambda_1^3$ and $\lambda_2\lambda_1$. A
  representative of $\lambda_1^3$ is the bracket~\eqref{longelement},
  see for example~\cite{Ellis-Mikhailov:2010}. More generally, each
  $\lambda_i$ corresponds to an operation on a simplicial group, with
  $\lambda_1$ corresponding (for $p=2$) to a simple bracketing
  $u\mapsto[s_0 u,s_1 u]$. Iterating it three times
  gives~\eqref{longelement}; see~\cite{Mikhailov:2018} for details.

  To show that $\widetilde\alpha_4$ represents the $4$-torsion, we
  observe first that it is a cycle, namely that it lies in
  $R_0\cap \cdots\cap R_5$, and secondly we show that, modulo
  $\gamma_9\gamma_8^2$, it represents the element $\lambda_2\lambda_1$
  of the simplicial Lie algebra, given as a sum
  \begin{align*}
    \omit\rlap{\kern-1cm$[[[z_0,z_3],[z_2,z_4]]+[[z_0,z_4],[z_2,z_3]]+[[z_3,z_4],[z_0,z_2]],$}\\
    & \kern1cm[[z_1,z_3],[z_2,z_4]]+[[z_1,z_4],[z_2,z_3]]+[[z_3,z_4],[z_1,z_2]]]\\
    ={}-{} & [[[z_3,z_1],[z_3,z_2]],[[z_4,z_0],[z_4,z_2]]]\\
    {}-{} & [[[z_4,z_1],[z_4,z_2]],[[z_3,z_0],[z_3,z_2]]]\\
    {}+{} & [[[z_3,z_2],[z_3,z_1]],[[z_4,z_2],[z_4,z_0]]]\\
    {}+{} & [[[z_4,z_2],[z_4,z_1]],[[z_3,z_2],[z_3,z_0]]]\\
    {}+{} & [[[z_4,z_2],[z_3,z_0]],[[z_4,z_2],[z_3,z_1]]]\\
    {}+{} & [[[z_4,z_1],[z_3,z_2]],[[z_4,z_2],[z_3,z_0]]]\\
    {}+{} & [[[z_4,z_2],[z_3,z_1]],[[z_4,z_0],[z_3,z_2]]]\\
    {}+{} & [[[z_4,z_0],[z_3,z_2]],[[z_4,z_1],[z_3,z_2]]]\\
    {}+{} & [[[z_4,z_2],[z_3,z_1]],[[z_4,z_3],[z_2,z_0]]]\\
    {}+{} & [[[z_4,z_3],[z_2,z_1]],[[z_4,z_2],[z_3,z_0]]]\\
    {}+{} & [[[z_4,z_3],[z_2,z_0]],[[z_4,z_1],[z_3,z_2]]]\\
    {}+{} & [[[z_4,z_0],[z_3,z_2]],[[z_4,z_3],[z_2,z_1]]]\\
    {}+{} & [[[z_4,z_3],[z_2,z_0]],[[z_4,z_3],[z_2,z_1]]].
  \end{align*}
\end{rem}

%%%%%%%%%%%%%%%%%%%%%%%%%%%%%%%%%%%%%%%%%%%%%%%%%%%%%%%%%%%%%%%%
\section{Examples}\label{ss:examples}
The homotopy classes presented above yielded with relatively little
computational effort Lie algebras and groups with $p$-torsion in some
high-degree dimension quotient. Using more computational resources, we
were able to find $p$-torsion in lower degree for $p=2$ and $p=3$.

A general simplification (see Propositions~\ref{prop:groupalphap}
and~\ref{prop:liealphap}) is that we can start by an element
$\alpha_p$ of degree $p$ and not $2p$, by writing generators $x_{i j}$
in place of $[x_i,x_j]$. Indeed all the computations that express
$\alpha_p$ as an symmetrized associative product actually take place
in $\Lie_p\Lie_2(\Z^{2p})\subset\Lie_{2p}(\Z^{2p})$. In fact, this
amounts to working in Milnor's simplicial construction $F[S^2]$, whose
geometric realization is $\Omega S^3$, and in its Lie analog
$L[S^2]$. Observe that, for spheres $S^d$ of dimension $d>3$, as well
as of Moore spaces, there is a description of homotopy groups as
centers of explicitly defined finitely generated
groups~\cite{Mikhailov-Wu:2013}. However, these groups are not as
easily defined as in the case of $S^2$, when we quotient by the
symmetric commutator. This is why we concentrated on $S^2\vee S^2$ in
this article.

\subsection{$p=2$}\label{ss:examples2}
The construction given in the proof of Theorem~\ref{thm:liemain} has
generators $x_0,x_1,x_2$ and $x_3\coloneqq-x_0-x_1-x_2$. The element
$\omega$ belongs to $\delta_{14}(A)\setminus\gamma_{14}(A)$. It is
possible to be a little bit more economical, by keeping the nilpotency
degrees of the $y_i$ more under control: the best we could achieve is
\begin{multline*}
  A=\langle x_0,x_1,x_2,x_3,y_0^{(1)},y_1^{(2)},y_2^{(2)},y_3^{(2)}\mid\\
  x_0+x_1+x_2+x_3=0,\;x_0=2^6y_0,\;2^6x_1=2^5y_1,\;2^5x_2=2^3y_2,\;2^3x_3=y_3\rangle
\end{multline*}
and the element $\omega=[[x_0,x_2],[x_1,x_2]]$. In that Lie algebra,
we have $\omega\in\delta_7(A)\setminus\gamma_7(A)$ and
$2\omega\in\gamma_7(A)$. This can been checked by hand, or computer using
the program \texttt{lienq} by Csaba Schneider~\cite{Schneider:1997},
or its improvement \texttt{anq}~\cite{ANQ}; see Appendix~\ref{ss:data}.

Rewriting $[x_i,x_j]$ as $x_{i j}$ allows more simplifications; we may consider general presentations of the form
\begin{align*}
  \MoveEqLeft[4] A=\langle x_{01},x_{02},x_{03},x_{12},x_{13},x_{23},y_{01}^{(2)},y_{02}^{(2)},y_{03}^{(2)},y_{12}^{(2)},y_{13}^{(2)},y_{23}^{(2)}\mid\\
  &x_{01}+x_{02}+x_{03} = -x_{01}+x_{12}+x_{13} = -x_{02}-x_{12}+x_{23}=0,\\
  &2^{a_{i j}} x_{i j}=2^{b_{i j}} y_{i j}\text{ for all }i,j\rangle.
\end{align*}
Suppose that the element $\omega$ is $2^n[x_{02}x_{12}]$ and we want
to show that it belongs to $\delta_4(A)\setminus\gamma_4(A)$. Using
the associative rewriting
$[x_{02},x_{12}]=x_{23}x_{01}+x_{12}x_{03}-x_{02}x_{13}$
from~\eqref{eq:exple2}, we will have $\omega\in\delta_4(A)$ as soon as
$A$ has relations of the form $2^n x_{23}=2^{a_{23}}y_{23}$ and
$2^{a_{23}}x_{01}=y_{01}$ and similarly for the other generators. The
condition $\omega\not\in\gamma_4(A)$ can be checked by a direct
calculation, e.g.\ using \texttt{anq}.

We may also replace the variables $x_{i j}$ by $2^{b_{i j}}x_{i j}$ for
well-chosen $b_{i j}$. This amounts, essentially, to letting the
$x_{i j}$ have degree less than $1$. For instance, replacing $x_{i j}$
by $2^{i+j}x_{i j}$ in the example above and simplifying somewhat, we
get
\begin{align*}
  \MoveEqLeft[4] A=\langle x_{01},x_{02},x_{12},y_{01}^{(2)},y_{02}^{(2)},y_{03}^{(2)},y_{12}^{(2)},y_{13}^{(2)},y_{23}^{(2)}\mid\\
  &2^2 x_{01}=y_{01},\;2^4 x_{02}=y_{02},\;2^6 x_{12}=y_{12},\\
  &2^6(-x_{01}-2x_{02})=2^6y_{03},\;2^5(x_{01}-4x_{12})=2^4y_{13},\;2^5(x_{02}+2x_{12})=2^2y_{23}\rangle
\end{align*}
with
$\omega\coloneqq
2^5[x_{02},x_{12}]\in\delta_4(A)\setminus\gamma_4(A)$. We have
$\omega\equiv
2^5[x_{01},x_{02}]+2^6[x_{01},x_{12}]+2^7[x_{02},x_{12}]$ modulo
$\gamma_4(A)$.  We may also choose $n\ge4$ and let $y_{0j}$ have
degree $n-2$ for all $j$, obtaining in this manner examples with
$2$-torsion in $\delta_n(A)/\gamma_n(A)$.

It is straightforward to convert the example above into a group: it will be
\begin{align*}
  \MoveEqLeft[4] G=\langle x_{01},x_{02},x_{12},y_{01}^{(2)},y_{02}^{(2)},y_{03}^{(2)},y_{12}^{(2)},y_{13}^{(2)},y_{23}^{(2)}\mid\\
  &x_{01}^4=y_{01},\;x_{02}^{16}=y_{02},\;x_{12}^{64}=y_{12},\\
  &x_{01}^{-64}x_{02}^{-128}=y_{03}^{64},\;x_{01}^{32}x_{12}^{-128}=y_{13}^{16},\;x_{02}^{32}x_{12}^{64}=y_{23}^4\rangle
\end{align*}
and the element $\omega=[e_0,e_1]^{32}[e_0,e_2]^{64}[e_1,e_2]^{128}$
belongs to $\delta_4(G)\setminus\gamma_4(G)$. Increasing the degree of
the $y_{0j}$ leads, for every $n\ge4$, to a group $G$ with $2$-torsion
in $\delta_n(G)/\gamma_n(G)$. This is essentially Rips's original
example~\eqref{eq:rips}, except that his example contains more
relations that make the group finite.

\subsection{$p=3$}
As in the $p=2$ case, we may construct a Lie algebra with
generators $x_{i j}$ as follows:
\begin{align*}
  \MoveEqLeft[4] A=\langle x_{i j},y_{i j}^{(i+j+1)}\text{ for }0\le i<j\le 5,\\
  & x_{01}+x_{02}+x_{03}+x_{04}+x_{05}=0,\\
  & {-}x_{01}+x_{12}+x_{13}+x_{14}+x_{15}=0,\\
  & {-}x_{02}-x_{12}+x_{23}+x_{24}+x_{25}=0,\\
  & {-}x_{03}-x_{13}-x_{23}+x_{34}+x_{35}=0,\\
  & {-}x_{04}-x_{14}-x_{24}-x_{34}+x_{45}=0,\\
  & 3^{i+j}x_{i j}=y_{i j}\text{ for }0\le i<j\le 5\rangle
\end{align*}
and the element $\omega=3^{15}([x_{04},x_{14},x_{23}]-[x_{04},x_{24},x_{13}]+[x_{04},x_{34},x_{12}]+[x_{14},x_{24},x_{03}]-[x_{14},x_{34},x_{02}]+[x_{24},x_{34},x_{01}])$ which belongs to $\delta_{18}(A)\setminus\gamma_{18}(A)$.

Again there is substantial flexibility in this example: the degrees of
the $y_{i j}$ may be adjusted, and the last relations may be changed to
$3^{a_{i j}}x_{i j}=3^{c_{i j}}y_{i j}$ for well-chosen
$a_{i j},c_{i j}$. The variables $x_{i j}$ themselves may be replaced by
$3^{b_{i j}}x_{i j}$ for well-chosen $b_{i j}$. Finally, some extra
linear conditions may be imposed on the variables, such as
$x_{02} = x_{13} = x_{15} = x_{24}=x_{34}=0$. After some
experimentation, we arrived at the following reasonably small
example:
\begin{equation}\label{eq:3lie}
  \begin{split}
    \MoveEqLeft[6] A=\langle e_0,e_1,e_2,e_3,y_i^{(2)}\text{ for }i\in\{0,\dots,3\},y_{i j}^{(3)}\text{ for }0\le i<j\le3\mid\\
    & 3^{2i}e_i=y_i,\;3^{12-i}e_j+3^{12-j}e_i=3^{12-2i-2j}y_{i j}\text{ for }(i,j)\in\{(0,1),(0,2),(1,3),(2,3)\}\rangle
\end{split}
\end{equation}
with $\omega=3^9[e_2,e_1,e_0]$.
\begin{prop}\label{prop:3lie}
  For the Lie ring $A$ defined in~\eqref{eq:3lie} we have
  $\omega\in\delta_7(A)\setminus\gamma_7(A)$ and
  $3\omega\in\gamma_7(A)$.
\end{prop}
\begin{proof}
  Expanding $\omega$ associatively, we get
  \[\omega = 3^9(e_0e_1e_2-e_0e_2e_1-e_1e_2e_0+e_2e_1e_0).\]
  We may rewrite it as
  \begin{align*}
    \omega &= -e_0(3^9e_2+3^{10}e_3)e_1 - e_1(3^9e_2+3^{10}e_3)e_0\\
           & \phantom{=-} + e_0(3^9e_1+3^{11}e_3)e_2 + e_2(3^9e_1+3^{11}e_3)e_0\\
           & \phantom{=-} + (3^{10}e_0+3^{12}e_2)e_3e_1 + e_1e_3(3^{10}e_0+3^{12}e_2)\\
           & \phantom{=-} - (3^{11}e_0+3^{12}e_1)e_3e_2 - e_2e_3(3^{11}e_0+3^{12}e_1).
  \end{align*}
  Each of the summands belongs to $\varpi^7(A)$: they are all products
  of $e_k$, $e_\ell$ and $3^{12-i}e_j+3^{12-j}e_i$ for some
  $\{i,j,k,\ell\}=\{0,1,2,3\}$. The binomial term equals
  $3^{12-2i-2j}y_{i j}=3^{2k+2\ell}y_{i j}$, so the summand is the
  product of $3^{2k}e_k$, $3^{2\ell}e_\ell$ and $y_{i j}$, namely the
  product of $y_k,y_\ell,y_{i j}$, of respective degrees $2,2,3$.

  To check that $\omega$ does not belong to $\gamma_7(A)$ but that
  $3\omega$ does, we compute nilpotent quotients of $A$. We did the
  calculation using two different programs: \texttt{lienq} by Csaba
  Schneider and \texttt{LieRing}~\cite{LieRing} for
  \texttt{GAP}~\cite{GAP4} by Willem de Graaf and Serena Cical\`o.
\end{proof}

In the next subsection, we give a direct proof that the associated
group has $3$-torsion in $\delta_7/\gamma_7$.

\subsection{\boldmath A small, finite $3$-group $G$ with $\delta_7(G)\neq\gamma_7(G)$}
We consider the group $G$ given by the presentation~\eqref{eq:3lie}, namely
\begin{equation}\label{eq:3group}
  \begin{split}
    \MoveEqLeft[10] G=\langle e_0,e_1,e_2,e_3,y_i^{(2)}\text{ for }i\in\{0,\dots,3\},y_{ij}^{(3)}\text{ for }0\le i<j\le 3\mid\\
    & e_i^{3^{2i}}=y_i,\;e_j^{3^{12-i}}e_i^{3^{12-j}}=y_{ij}^{3^{12-2i-2j}}\text{ for }(i,j)\in\{(0,1),(0,2),(1,3),(2,3)\}\rangle
  \end{split}
\end{equation}
with $\omega=[e_2,e_1,e_0]^{3^9}$.

\begin{prop}
  In the group defined by~\eqref{eq:3group} we have
  $\omega\in\delta_7(G)$.
\end{prop}
\begin{proof}
  We will use the following well-known identity, which holds for any
  element $x\in G$ and $d\geq 2$:
  \begin{equation}\label{binomials}
    x^d-1=\sum_{k=1}^d\binom{d}{k}(x-1)^k.
  \end{equation}
  We compute modulo $\varpi^7(G)$, and from now on write $\equiv$ to
  mean equivalence modulo $\varpi^7(G)$. We get
  \begin{align*}
    \omega-1 &\equiv 3^9([e_2,e_1,e_0]-1)\text{ since }[e_2,e_1,e_0]\in \gamma_4(G)\\
             &= 3^9[e_1,e_2]e_0^{-1}\big(([e_2,e_1]-1)(e_0-1)-(e_0-1)([e_2,e_1]-1)\big)\\
    \begin{split}
      =3^9[e_1,e_2]e_0^{-1}\big(e_2^{-1}e_1^{-1}((e_2-1)(e_1-1)-(e_1-1)(e_2-1))(e_0-1)\\
      -(e_0-1)e_2^{-1}e_1^{-1}((e_2-1)(e_1-1)-(e_1-1)(e_2-1))\big);
    \end{split}\\
    \intertext{and since $3^9$ is divisible by the product of exponent
    of $e_0,e_1,e_2$ modulo $\gamma_2(G)$,}
    \begin{split}\equiv 3^9\big((e_0-1)(e_1-1)(e_2-1)-(e_0-1)(e_2-1)(e_1-1)\\
      {}-(e_1-1)(e_2-1)(e_0-1)+(e_2-1)(e_1-1)(e_0-1)\big).
    \end{split}
  \end{align*}
  Let us write $f_i\coloneqq e_i-1$ for $i\in\{0,1,2,3\}$. Then, as in
  the proof of Proposition~\ref{prop:3lie}, we have
  \begin{align*}
    \omega-1 &\equiv 3^9(f_0f_1f_2-f_0f_2f_1-f_1f_2f_0+f_2f_1f_0)\\
             &= -f_0(3^9f_2+3^{10}f_3)f_1 - f_1(3^9f_2+3^{10}f_3)f_0\\
           & \phantom{=-} + f_0(3^9f_1+3^{11}f_3)f_2 + f_2(3^9f_1+3^{11}f_3)f_0\\
           & \phantom{=-} + (3^{10}f_0+3^{12}f_2)f_3f_1 + f_1f_3(3^{10}f_0+3^{12}f_2)\\
           & \phantom{=-} - (3^{11}f_0+3^{12}f_1)f_3f_2 - f_2f_3(3^{11}f_0+3^{12}f_1).
  \end{align*}
  Next, using~\eqref{binomials} we have
  \begin{align*}
    e_i^{3^{12-j}}e_j^{3^{12-i}}-1 &= (e_i^{3^{12-j}}-1)+(e_j^{3^{12-i}}-1)+(e_i^{3^{12-j}}-1)(e_j^{3^{12-i}}-1)\\
    &= 3^{12-j}f_i+\binom{3^{12-j}}{2}f_i^2+\binom{3^{12-j}}{3}f_i^3+\cdots\\
    &\phantom{=} + 3^{12-i}f_j+\binom{3^{12-i}}{2}f_j^2+\binom{3^{12-i}}{3}f_j^3+\cdots\\
    &\phantom{=} + 3^{24-i-j}f_i f_j+\cdots\\
    =y_{i j}^{3^{12-2i-2j}}-1 &= 3^{12-2i-2j}(y_{i j}-1)+\binom{3^{12-2i-2j}}{2}(y_{i j}-1)^2+\cdots
  \end{align*}
  Again using~\eqref{binomials}, the relations $e_i^{3^{2i}}=y_i$
  imply $3^{2i}f_i\in\varpi^2$.  Now $3^{12-2i-2j}$ divides
  $\binom{3^{12-j}}{3}/3^{2i}$, so
  $\binom{3^{12-j}}{2}f_i^2\in 3^{12-2i-2j}\varpi^3$. Similarly,
  $3^{12-2i-2j}$ divides $\binom{3^{12-j}}{3}$ and
  $\binom{3^{12-j}}{4}$ so all terms with a binomial co\"efficient
  belong to $3^{12-2i-2j}\varpi^3+\varpi^5$. The same holds for all
  terms in the last two rows. We therefore have
  \[3^{12-j}f_i+3^{12-i}f_j\in 3^{12-2i-2j}\varpi^3+\varpi^5.
  \]
  We note $3^{12-2i-2j}=3^{2k+2\ell}$ whenever
  $\{i,j,k,\ell\}=\{0,1,2,3\}$. Returning to computations modulo
  $\varpi^7$, we consider a typical summand
  $f_k(3^{12-j}f_i+3^{12-i}f_j)f_\ell$ in our decomposition of
  $\omega-1$. We write $3^{12-j}f_i+3^{12-i}f_j=3^{12-2i-2j}u+v$ with
  $u\in\varpi^3,v\in\varpi^5$ to get
  \[f_k(3^{12-j}f_i+3^{12-i}f_j)f_\ell =f_k(3^{2k+2\ell}u+v)f_\ell
    =(3^{2k}f_k)u(3^{2\ell}f_\ell)+f_k v f_\ell
  \]
  where each summand belongs to $\varpi^7$.
\end{proof}

\begin{prop}
  In the group defined by~\eqref{eq:3group}, the element $\omega$
  defined above does not belong to $\gamma_7(G)$, but its cube does.
\end{prop}
\begin{proof}
  The proof is computer-assisted. It suffices to exhibit a quotient
  $\overline G$ of $G$ in which the image of $\omega$ does not belong
  to $\gamma_7(\overline G)$ but its cube does, and we shall exhibit a
  finite $3$-group as quotient.

  To make the computations more manageable, we replace the generators
  $y_i$ and $y_{i j}$ by generators $z_0,\dots,z_3$, and impose the
  choices
  \begin{xalignat*}{4}
    y_0 &= [z_0,z_1], & y_1 &= [z_0,z_2], & y_2 &= [z_0,z_3], & y_3 &= [z_1,z_2],\\
    y_{01} &= 1, & y_{02} &= [z_1,z_3,z_3], & y_{13} &= [z_1,z_3,z_1], & y_{23} &= [z_1,z_3,z_0].
  \end{xalignat*}
  
  In this manner, we obtain an $8$-generated group
  $\langle e_0,\dots,e_3,z_0,\dots,z_3\rangle$. We next impose extra
  commutation relations:
  $[e_2,z_2],[e_3,z_2],[e_1,z_3],[e_2,z_3],[e_3,z_3]$.

  We compute a basis of left-normed commutators of length at most $6$
  in that group; notice that $\omega$ may be expressed as
  $[z_3,z_2,z_3,z_1,z_1,e_3]^{3^5}$, and impose extra relations making
  $\gamma_6$ cyclic and central.

  The resulting finite presentation may be fed to the program
  \texttt{pq} by Eamonn O'Brien~\cite{PQ}, to compute the maximal
  quotient of $3$-class $17$. This is a group of order $3^{3996}$, and
  can (barely) be loaded in the computer algebra system
  \texttt{GAP}~\cite{GAP4} so as to check (for safety) that the
  relations of $G$ hold, and that the element $\omega$ has a
  non-trivial image in it.

  Finally, the order of the group may be reduced by iteratively
  quotienting by maximal subgroups of the centre that do not contain
  $\omega$.
\end{proof}

The resulting group, which is the minimal-order $3$-group with
non-trivial dimension quotient that we could obtain, has order
$3^{494}$.

It may be loaded in any \texttt{GAP} distribution by downloading the
ancillary file \verb+3group.gap+ to the current directory and running
\verb+Read("3group.gap");+ in a \texttt{GAP} session.

%%%%%%%%%%%%%%%%%%%%%%%%%%%%%%%%%%%%%%%%%%%%%%%%%%%%%%%%%%%%%%%%
\section{Acknowledgments}
The authors are grateful to Jacques Darn\'e for having pointed out a
mistake in a previous version of this text, which led us to simplify
some proofs.

%%%%%%%%%%%%%%%%%%%%%%%%%%%%%%%%%%%%%%%%%%%%%%%%%%%%%%%%%%%%%%%%
\appendix\section{The program \texttt{anq}}\label{ss:data}
We have developed a power computer program to explore nilpotent
quotients of finitely presented Lie rings. It is freely available at
\url{https://github.com/laurentbartholdi/anq}. The first example in~\S\ref{ss:examples2} is entered by putting the following in a file:
\begin{verbatim}
< x0, x1, x2, x3, omega, omega2, y0; y1; y2; y3 | x0+x1+x2+x3,
  x0 = 2^7*y0, 2^1*x1 = y1, 2^2*x2 = y2, 2^4*x3 = y3,
  omega := [[x0,x2],[x1,x2]], omega2 := 2*omega >
\end{verbatim}
Without going into details: generators are listed, separated by commas
(\verb+,+) or semicolons (\verb+;+); the degree of a generator is one
more than the number of preceding semicolons. Relations are then
listed after the \verb+|+. The last two relations are \emph{aliases}:
they define a generator in terms of previously-listed
generators. (Being an alias merely speeds up the program).

\texttt{anq} supports a variety of rings; in particular, finite
rings $\Z/p^n\Z$ in which arithmetic is very fast; fixed-precision
integers (which abort under overflow); and arbitrary-precision
integers (which tend to be quite slow). Since for the example above we
are interested in $2$-torsion, we compile an executable for $p=2$ and
$n$ large:
\begin{verbatim}
% make nq_l_2_64
\end{verbatim}
The ``\verb+l+'' means ``Lie algebra'', and $p$ and $n$ are given
separated by underscores (\verb+_+). Replacing ``\verb+l+'' by
``\verb+g+'' would compile a group quotient program. If the presentation above was saved in file \verb+twotorsion+, we could then invoke
\begin{verbatim}
% ./nq_l_2_64 -N4 -W9 ./twotorsion | grep omega
#      omega |--> 256*a149 + 4*a150 + 252*a152 + 1*a156 + 2*a157 + 256*a159 + 1*a161
#     omega2 |-->
\end{verbatim}
This tells us that, in the maximal quotient of the given Lie algebra
of nilpotency class $4$ and maximal degree $9$, the element $\omega$
is non-trivial but its double is trivial. Note that the nilpotency
class option ``\verb+-N4+'' forces all five-fold iterated commutators
to vanish, and also serves as a speedup; since the generators
\verb+y1,y2,y3+ have degree $2,3,4$ respectively, the effective
nilpotency class of the quotient is at most $1+2+3+4=10$, if each
\verb+yi+ is written as an ($i+1$)-fold iterated commutator of
degree-$1$ generators.

We have in this manner verified that $\omega$ does not belong to
$\gamma_{10}$, but that $2\omega$ does belong to $\gamma_{10}$. It
remains to check, by hand, that $\omega$ belongs to $\delta_{10}$ to
conclude that indeed the example above has $2$-torsion in
$\delta_{10}/\gamma_{10}$.

Note that the same verification could have been made by computing with
coefficients $\Z/2^{15}$; but the answer with coefficients $\Z/2^{14}$
would have been inconclusive.

The Lie algebra example~\eqref{eq:3lie} was also checked using
\texttt{anq}, as follows:
\begin{verbatim}
% make nq_l_3_38
% echo '< e0, e1, e2, e3, omega, omega3; y0, y1, y2, y3; y00, y01, y02, y03, y12, y13, y23 |
  3^0*e0 = y0, 3^2*e1 = y1, 3^4*e2 = y2, 3^6*e3 = y3,
  3^12*e1 + 3^11*e0 = 3^10*y01,
  3^12*e2 + 3^10*e0 = 3^8*y02,
  3^11*e3 + 3^9*e1 = 3^4*y13,
  3^10*e3 + 3^9*e2 = 3^2*y23,
  omega := 3^9*[e2,e1,e0], omega3 := 3*omega >' | ./nq_l_3_38 -W6 -N3 | grep omega
#      omega |--> 2*a323
#     omega3 |-->
\end{verbatim}


\begin{bibdiv}
\begin{biblist}
\bib{ANQ}{manual}{
     title={ANQ --- all nilpotent quotients, version 5.0.0},
     author={Laurent Bartholdi},
      date={2021},
       url={\texttt{http://www.github.com/laurentbartholdi/anq}},
}

\bib{Bartholdi-Passi:2015}{article}{
   author={Bartholdi, Laurent},
   author={Passi, Inder Bir S.},
   title={Lie dimension subrings},
   journal={Internat. J. Algebra Comput.},
   volume={25},
   date={2015},
   number={8},
   pages={1301--1325},
   issn={0218-1967},
   review={\MR{3438163}},
   doi={10.1142/S0218196715500423},
}

\bib{Bott-Samelson:1953}{article}{
   author={Bott, Raoul H.},
   author={Samelson, Hans},
   title={On the Pontryagin product in spaces of paths},
   journal={Comment. Math. Helv.},
   volume={27},
   date={1953},
   pages={320--337 (1954)},
   issn={0010-2571},
   review={\MR{60233}},
   doi={10.1007/BF02564566},
}

\bib{Bridson-Haefliger:1999}{book}{
    author={Bridson, Martin~R.},
    author={Haefliger, Andr\'e},
     title={Metric spaces of non-positive curvature},
 publisher={Springer-Verlag},
   address={Berlin},
      date={1999},
      ISBN={3-540-64324-9},
    review={\MR{2000k:53038}},
}
  
\bib{Bokut-Kukin:1994}{book}{
   author={Bokut\cprime , Leonid A.},
   author={Kukin, Georgii P.},
   title={Algorithmic and combinatorial algebra},
   series={Mathematics and its Applications},
   volume={255},
   publisher={Kluwer Academic Publishers Group, Dordrecht},
   date={1994},
   pages={xvi+384},
   isbn={0-7923-2313-0},
   review={\MR{1292459}},
   doi={10.1007/978-94-011-2002-9},
}

\bib{BM}{article}{
   author={Breen, Lawrence},
   author={Mikhailov, Roman},
   title={Derived functors of nonadditive functors and homotopy theory},
   journal={Algebr. Geom. Topol.},
   volume={11},
   date={2011},
   number={1},
   pages={327--415},
   issn={1472-2747},
   review={\MR{2764044}},
   doi={10.2140/agt.2011.11.327},
}

\bib{6A}{article}{
   author={Bousfield, Aldridge K.},
   author={Curtis, Edward B.},
   author={Kan, Dani\"el M.},
   author={Quillen, Daniel G.},
   author={Rector, David L.},
   author={Schlesinger, James W.},
   title={The ${\rm mod-}p$ lower central series and the Adams spectral
   sequence},
   journal={Topology},
   volume={5},
   date={1966},
   pages={331--342},
   issn={0040-9383},
   review={\MR{0199862}},
   doi={10.1016/0040-9383(66)90024-3},
}

\bib{Cartier:1958}{article}{
   author={Cartier, Pierre},
   title={Remarques sur le th\'eor\`eme de Birkhoff-Witt},
   language={French},
   journal={Ann. Scuola Norm. Sup. Pisa (3)},
   volume={12},
   date={1958},
   pages={1--4},
   review={\MR{0098121 (20 \#4583)}},
}

\bib{Cohen:1993}{article}{
   author={Cohen, Frederick R.},
   title={On combinatorial group theory in homotopy},
   conference={
      title={Homotopy theory and its applications},
      address={Cocoyoc},
      date={1993},
   },
   book={
      series={Contemp. Math.},
      volume={188},
      publisher={Amer. Math. Soc., Providence, RI},
   },
   date={1995},
   pages={57--63},
   review={\MR{1349129}},
   doi={10.1090/conm/188/02233},
}


\bib{CohenWu:2001}{article}{
   author={Cohen, Frederick R.},
   author={Wu, Jie},
   title={On braid groups, free groups, and the loop space of the 2-sphere},
   conference={
      title={Categorical decomposition techniques in algebraic topology},
      address={Isle of Skye},
      date={2001},
   },
   book={
      series={Progr. Math.},
      volume={215},
      publisher={Birkh\"{a}user, Basel},
   },
   date={2004},
   pages={93--105},
   review={\MR{2039761}},
}

\bib{CohenWu:2011}{article}{
   author={Cohen, Frederick R.},
   author={Wu, Jie},
   title={Artin's braid groups, free groups, and the loop space of the
   2-sphere},
   journal={Q. J. Math.},
   volume={62},
   date={2011},
   number={4},
   pages={891--921},
   issn={0033-5606},
   review={\MR{2853222}},
   doi={10.1093/qmath/haq010},
}

\bib{Cohn:1952}{article}{
   author={Cohn, Paul M.},
   title={Generalization of a theorem of Magnus},
   journal={Proc. London Math. Soc. (3)},
   volume={2},
   date={1952},
   pages={297--310},
   issn={0024-6115},
   review={\MR{0051834}},
   doi={10.1112/plms/s3-2.1.297},
}

\bib{Curtis:1965}{article}{
   author={Curtis, Edward B.},
   title={Some relations between homotopy and homology},
   journal={Ann. of Math. (2)},
   volume={82},
   date={1965},
   pages={386--413},
   issn={0003-486X},
   review={\MR{184231}},
   doi={10.2307/1970703},
 }
 
\bib{DP}{article}{
   author={Dold, Albrecht},
   author={Puppe, Dieter},
   title={Homologie nicht-additiver Funktoren. Anwendungen},
   language={German, with French summary},
   journal={Ann. Inst. Fourier Grenoble},
   volume={11},
   date={1961},
   pages={201--312},
   issn={0373-0956},
   review={\MR{0150183}},
}

\bib{Ellis-Mikhailov:2010}{article}{
   author={Ellis, Graham},
   author={Mikhailov, Roman},
   title={A colimit of classifying spaces},
   journal={Adv. Math.},
   volume={223},
   date={2010},
   number={6},
   pages={2097--2113},
   issn={0001-8708},
   review={\MR{2601009}},
   doi={10.1016/j.aim.2009.11.003},
}

\bib{Fox:fdc1}{article}{
   author={Fox, Ralph~H.},
   title={Free differential calculus. I. Derivation in the free group ring},
   journal={Ann. of Math. (2)},
   volume={57},
   date={1953},
   pages={547\ndash 560},
   issn={0003-486X},
   review={\MR{0053938 (14,843d)}},
}

\bib{Fox:fdc4}{article}{
   author={Chen, Kuo Tsai},
   author={Fox, Ralph~H.},
   author={Lyndon, Roger~C.},
   title={Free differential calculus. IV. The quotient groups of the lower
   central series},
   journal={Ann. of Math. (2)},
   volume={68},
   date={1958},
   pages={81\ndash 95},
   issn={0003-486X},
   review={\MR{0102539 (21 \#1330)}},
}

\bib{GAP4}{manual}{
     title={GAP --- Groups, Algorithms, and Programming, Version 4.8.10},
     label={GAP16},
    author={The GAP~Group},
      date={2018},
       url={\texttt{http://www.gap-system.org}},
}

\bib{Geoghegan:2008}{book}{
   author={Geoghegan, Ross},
   title={Topological methods in group theory},
   series={Graduate Texts in Mathematics},
   volume={243},
   publisher={Springer, New York},
   date={2008},
   pages={xiv+473},
   isbn={978-0-387-74611-1},
   review={\MR{2365352}},
   doi={10.1007/978-0-387-74614-2},
}

\bib{Gray:1969}{article}{
   author={Gray, Brayton I.},
   title={On the sphere of origin of infinite families in the homotopy
   groups of spheres},
   journal={Topology},
   volume={8},
   date={1969},
   pages={219--232},
   issn={0040-9383},
   review={\MR{245008}},
   doi={10.1016/0040-9383(69)90012-3},
}
 
\bib{Gromov:1991}{article}{
   author={Gromov, Mikhail},
   title={Asymptotic invariants of infinite groups},
   conference={
      title={Geometric group theory, Vol. 2},
      address={Sussex},
      date={1991},
   },
   book={
      series={London Math. Soc. Lecture Note Ser.},
      volume={182},
      publisher={Cambridge Univ. Press, Cambridge},
   },
   date={1993},
   pages={1--295},
   review={\MR{1253544}},
}
 
\bib{Gupta-Kuzmin:1992}{article}{
   author={Gupta, Narain},
   author={Kuz\cprime min, Yuri},
   title={On varietal quotients defined by ideals generated by Fox
   derivatives},
   journal={J. Pure Appl. Algebra},
   volume={78},
   date={1992},
   number={2},
   pages={165--172},
   issn={0022-4049},
   review={\MR{1161340}},
   doi={10.1016/0022-4049(92)90094-V},
}

 \bib{Gupta:1987}{book}{
   author={Gupta, Narain D.},
   title={Free group rings},
   series={Contemporary Mathematics},
   volume={66},
   publisher={American Mathematical Society, Providence, RI},
   date={1987},
   pages={xii+129},
   isbn={0-8218-5072-5},
   review={\MR{895359}},
   doi={10.1090/conm/066},
}

\bib{Gupta:1988}{article}{
   author={Gupta, Narain D.},
   title={Dimension subgroups of metabelian $p$-groups},
   journal={J. Pure Appl. Algebra},
   volume={51},
   date={1988},
   number={3},
   pages={241--249},
   issn={0022-4049},
   review={\MR{946575}},
   doi={10.1016/0022-4049(88)90063-1},
}

\bib{Gupta:1991}{article}{
   author={Gupta, Narain D.},
   title={A solution of the dimension subgroup problem},
   journal={J. Algebra},
   volume={138},
   date={1991},
   number={2},
   pages={479--490},
   issn={0021-8693},
   review={\MR{1102820}},
   doi={10.1016/0021-8693(91)90182-8},
}

\bib{Gupta:1996}{article}{
   author={Gupta, Narain D.},
   title={Lectures on dimension subgroups},
   note={Group rings and related topics (Portuguese) (S\~ao Paulo, 1995)},
   journal={Resenhas},
   volume={2},
   date={1996},
   number={3},
   pages={263--273},
   issn={0104-3854},
   review={\MR{1410298}},
}
 
\bib{Gupta:2002}{article}{
   author={Gupta, Narain D.},
   title={The dimension subgroup conjecture holds for odd order groups},
   journal={J. Group Theory},
   volume={5},
   date={2002},
   number={4},
   pages={481\ndash 491},
   issn={1433-5883},
   review={\MR{1931371 (2003m:20019)}},
}

\bib{Hilton:1955}{article}{
   author={Hilton, Peter J.},
   title={On the homotopy groups of the union of spheres},
   journal={J. London Math. Soc.},
   volume={30},
   date={1955},
   pages={154--172},
   issn={0024-6107},
   review={\MR{68218}},
   doi={10.1112/jlms/s1-30.2.154},
}
		
\bib{HuangWu:2020}{article}{
   author={Huang, Ruizhi},
   author={Wu, Jie},
   title={Combinatorics of double loop suspensions, evaluation maps and
   Cohen groups},
   journal={J. Math. Soc. Japan},
   volume={72},
   date={2020},
   number={3},
   pages={847--889},
   issn={0025-5645},
   review={\MR{4125848}},
   doi={10.2969/jmsj/81678167},
}

\bib{jennings:gpringnilp}{article}{
    author={Jennings, Stephen~A.},
     title={The group ring of a class of infinite nilpotent groups},
      date={1955},
   journal={Canad. J. Math.},
    volume={7},
     pages={169\ndash 187},
}

\bib{lazard:nilp}{article}{
    author={Lazard, Michel},
     title={Sur les groupes nilpotents et les anneaux de Lie},
      date={1953},
   journal={Ann. \'Ecole Norm. Sup. (3)},
    volume={71},
     pages={101\ndash 190},
}

\bib{Leibowitz:1972}{thesis}{
  author={Leibowitz, D.},
  title={The $E^1$ term of the lower central series spectral sequence for the homotopy of spaces},
  date={1972},
  place={Brandeis University},
  type={Ph.D. thesis},
}

\bib{LiWu:2011}{article}{
   author={Li, J. Y.},
   author={Wu, Jie},
   title={On symmetric commutator subgroups, braids, links and homotopy
   groups},
   journal={Trans. Amer. Math. Soc.},
   volume={363},
   date={2011},
   number={7},
   pages={3829--3852},
   issn={0002-9947},
   review={\MR{2775829}},
   doi={10.1090/S0002-9947-2011-05339-0},
}

\bib{LieRing}{manual}{
     title={LieRing --- a GAP package, version 2.3},
     author={de Graaf, Willem},
     author={Cical\`o, Serena},
      date={2016},
       url={\texttt{http://www.science.unitn.it/~degraaf/liering.html}},
}

\bib{MR1470727}{article}{
   author={Lin, Xiao-Song},
   title={Power series expansions and invariants of links},
   conference={
      title={Geometric topology},
      address={Athens, GA},
      date={1993},
   },
   book={
      series={AMS/IP Stud. Adv. Math.},
      volume={2},
      publisher={Amer. Math. Soc., Providence, RI},
   },
   date={1997},
   pages={184--202},
   review={\MR{1470727}},
}

 \bib{Losey:1960}{article}{
   author={Losey, Gerald},
   title={On dimension subgroups},
   journal={Trans. Amer. Math. Soc.},
   volume={97},
   date={1960},
   pages={474--486},
   issn={0002-9947},
   review={\MR{0148754}},
   doi={10.2307/1993383},
}

\bib{Lyndon-Schupp:1970}{book}{
    author={Lyndon, Roger~C.},
    author={Schupp, Paul~E.},
     title={Combinatorial group theory},
 publisher={Springer-Verlag},
      date={1970},
      ISBN={ISBN 3-540-07642-5},
}
    
\bib{Magnus:1935}{article}{
   author={Magnus, Wilhelm},
   title={Beziehungen zwischen Gruppen und Idealen in einem speziellen Ring},
   language={German},
   journal={Math. Ann.},
   volume={111},
   date={1935},
   number={1},
   pages={259--280},
   issn={0025-5831},
   review={\MR{1512992}},
   doi={10.1007/BF01472217},
}

\bib{Magnus:1937}{article}{
   author={Magnus, Wilhelm},
   title={\"Uber Beziehungen zwischen h\"oheren Kommutatoren},
   language={German},
   journal={J. Reine Angew. Math.},
   volume={177},
   date={1937},
   pages={105--115},
   issn={0075-4102},
   review={\MR{1581549}},
   doi={10.1515/crll.1937.177.105},
}

\bib{MKS}{book}{
   author={Magnus, Wilhelm},
   author={Karrass, Abraham},
   author={Solitar, Donald},
   title={Combinatorial group theory: Presentations of groups in terms of
   generators and relations},
   publisher={Interscience Publishers [John Wiley \& Sons, Inc.], New
   York-London-Sydney},
   date={1966},
   pages={xii+444},
   review={\MR{0207802}},
}

\bib{May:1967}{book}{
   author={May, J. Peter},
   title={Simplicial objects in algebraic topology},
   series={Van Nostrand Mathematical Studies, No. 11},
   publisher={D. Van Nostrand Co., Inc., Princeton, N.J.-Toronto,
   Ont.-London},
   date={1967},
   pages={vi+161},
   review={\MR{0222892}},
}

\bib{Mikhailov-Passi:2009}{book}{
   author={Mikhailov, Roman},
   author={Passi, Inder Bir S.},
   title={Lower central and dimension series of groups},
   series={Lecture Notes in Mathematics},
   volume={1952},
   publisher={Springer-Verlag, Berlin},
   date={2009},
   pages={xxii+346},
   isbn={978-3-540-85817-1},
   review={\MR{2460089}},
}

\bib{MPW}{article}{
   author={Mikhailov, Roman},
   author={Passi, Inder Bir S.},
   author={Wu, Jie},
   title={Symmetric ideals in group rings and simplicial homotopy},
   journal={J. Pure Appl. Algebra},
   volume={215},
   date={2011},
   number={5},
   pages={1085--1092},
   issn={0022-4049},
   review={\MR{2747240}},
   doi={10.1016/j.jpaa.2010.07.013},
}

\bib{Mikhailov-Wu:2013}{article}{
   author={Mikhailov, Roman},
   author={Wu, Jie},
   title={Combinatorial group theory and the homotopy groups of finite
   complexes},
   journal={Geom. Topol.},
   volume={17},
   date={2013},
   number={1},
   pages={235--272},
   issn={1465-3060},
   review={\MR{3035327}},
   doi={10.2140/gt.2013.17.235},
}

\bib{Mikhailov:2018}{article}{
  author={Mikhailov, Roman},
  title={Homotopy theory of Lie functors},
  eprint={arXiv:1808.00681},
  date={2018},
}

\bib{PQ}{article}{
   author={O'Brien, Eamonn A.},
   title={The $p$-group generation algorithm},
   note={Computational group theory, Part 1},
   journal={J. Symbolic Comput.},
   volume={9},
   date={1990},
   number={5-6},
   pages={677--698},
   issn={0747-7171},
   review={\MR{1075431}},
   doi={10.1016/S0747-7171(08)80082-X},
}

\bib{Passi:1968}{article}{
   author={Passi, Inder Bir S.},
   title={Dimension subgroups},
   journal={J. Algebra},
   volume={9},
   date={1968},
   pages={152--182},
   issn={0021-8693},
   review={\MR{0231916}},
   doi={10.1016/0021-8693(68)90018-5},
}

\bib{Passi:1979}{book}{
   author={Passi, Inder Bir S.},
   title={Group rings and their augmentation ideals},
   series={Lecture Notes in Mathematics},
   volume={715},
   publisher={Springer, Berlin},
   date={1979},
   pages={vi+137},
   isbn={3-540-09254-4},
   review={\MR{537126}},
}

\bib{Prufer:1923}{article}{
   author={Pr\"{u}fer, Heinz},
   title={Untersuchungen \"{u}ber die Zerlegbarkeit der abz\"{a}hlbaren prim\"{a}ren
   Abelschen Gruppen},
   language={German},
   journal={Math. Z.},
   volume={17},
   date={1923},
   number={1},
   pages={35--61},
   issn={0025-5874},
   review={\MR{1544601}},
   doi={10.1007/BF01504333},
}

\bib{Rips:1972}{article}{
   author={Rips, Eliyahu},
   title={On the fourth integer dimension subgroup},
   journal={Israel J. Math.},
   volume={12},
   date={1972},
   pages={342\ndash 346},
   issn={0021-2172},
   review={\MR{0314988 (47 \#3537)}},
}

\bib{Schlesinger}{article}{
   author={Schlesinger, James W.},
   title={The semi-simplicial free Lie ring},
   journal={Trans. Amer. Math. Soc.},
   volume={122},
   date={1966},
   pages={436--442},
   issn={0002-9947},
   review={\MR{0199861}},
   doi={10.2307/1994559},
}

\bib{Schneider:1997}{article}{
   author={Schneider, Csaba},
   title={Computing nilpotent quotients in finitely presented Lie rings},
   journal={Discrete Math. Theor. Comput. Sci.},
   volume={1},
   date={1997},
   number={1},
   pages={1--16},
   issn={1365-8050},
   review={\MR{1471347}},
}

\bib{Serre}{article}{
   author={Serre, Jean-Pierre},
   title={Homologie singuli\`ere des espaces fibr\'es. Applications},
   language={French},
   journal={Ann. of Math. (2)},
   volume={54},
   date={1951},
   pages={425--505},
   issn={0003-486X},
   review={\MR{0045386}},
   doi={10.2307/1969485},
}

\bib{Sicking:2020}{thesis}{
  author={Sicking, Thomas},
  title={Dimension subrings of Lie rings},
  date={2020},
  place={G\"ottingen University},
  type={Ph.D. thesis},
}

\bib{Sjogren:1979}{article}{
   author={Sjogren, Jon A.},
   title={Dimension and lower central subgroups},
   journal={J. Pure Appl. Algebra},
   volume={14},
   date={1979},
   number={2},
   pages={175--194},
   issn={0022-4049},
   review={\MR{524186}},
   doi={10.1016/0022-4049(79)90006-9},
}

\bib{Stallings:1975}{article}{
   author={Stallings, John R.},
   title={Quotients of the powers of the augmentation ideal in a group ring},
   conference={
      title={Knots, groups, and $3$-manifolds (Papers dedicated to the
      memory of R. H. Fox)},
   },
   book={
      publisher={Princeton Univ. Press, Princeton, N.J.},
   },
   date={1975},
   pages={101--118. Ann. of Math. Studies, No. 84},
   review={\MR{0379685}},
}

\bib{Tahara:1977a}{article}{
   author={Tahara, Ken-Ichi},
   title={The fourth dimension subgroups and polynomial maps},
   journal={J. Algebra},
   volume={45},
   date={1977},
   number={1},
   pages={102--131},
   issn={0021-8693},
   review={\MR{0432763}},
   doi={10.1016/0021-8693(77)90365-9},
}

\bib{Tahara:1981}{article}{
   author={Tahara, Ken-Ichi},
   title={The augmentation quotients of group rings and the fifth dimension
   subgroups},
   journal={J. Algebra},
   volume={71},
   date={1981},
   number={1},
   pages={141--173},
   issn={0021-8693},
   review={\MR{627430}},
   doi={10.1016/0021-8693(81)90111-3},
 }
 
\bib{Witt:1937}{article}{
   author={Witt, Ernst},
   title={Treue Darstellung Liescher Ringe},
   language={German},
   journal={J. Reine Angew. Math.},
   volume={177},
   date={1937},
   pages={152--160},
   issn={0075-4102},
   review={\MR{1581553}},
   doi={10.1515/crll.1937.177.152},
}

\bib{Wu:2001}{article}{
   author={Wu, Jie},
   title={Combinatorial descriptions of homotopy groups of certain spaces},
   journal={Math. Proc. Cambridge Philos. Soc.},
   volume={130},
   date={2001},
   number={3},
   pages={489--513},
   issn={0305-0041},
   review={\MR{1816806}},
   doi={10.1017/S030500410100487X},
}
		
		
\bib{Wu:2010}{article}{
   author={Wu, Jie},
   title={Simplicial objects and homotopy groups},
   conference={
      title={Braids},
   },
   book={
      series={Lect. Notes Ser. Inst. Math. Sci. Natl. Univ. Singap.},
      volume={19},
      publisher={World Sci. Publ., Hackensack, NJ},
   },
   date={2010},
   pages={31--181},
   review={\MR{2605306}},
   doi={10.1142/9789814291415\_0002},
}		
\end{biblist}
\end{bibdiv}
\end{document}